\documentclass[11pt]{amsart}

 \usepackage[utf8]{inputenc}
\usepackage{enumitem} 
\usepackage{tikz} 
\usepackage{tikz-cd} 
\usepackage{hyperref} 
\usepackage[capitalise]{cleveref} 
\usepackage{amsmath}
\usepackage{bbm}
\usepackage{mathtools}
\usepackage{geometry}
\usepackage{dsfont} 
\usepackage{comment}
\usepackage{ifthen} 
\usepackage{latexsym}

\usepackage[
textwidth=3cm,
textsize=small,
colorinlistoftodos]
{todonotes} 

\newtheorem{thm}{Theorem}[section] 
\newtheorem{cor}[thm]{Corollary}
\newtheorem{lem}[thm]{Lemma}
\newtheorem{prop}[thm]{Proposition}

\newtheorem*{theorem*}{Theorem}
\newtheorem*{cor*}{Corollary}

\theoremstyle{definition}
\newtheorem{defn}[thm]{Definition}
\newtheorem{notation}[thm]{Notation}
\newtheorem{convention}[thm]{Convention}
\newtheorem{indhyp}[thm]{Inductive Hypothesis}
\newtheorem{ClosureProp}[thm]{Closure Properties}

\theoremstyle{remark}
\newtheorem{rmk}[thm]{Remark}
\newtheorem{ex}[thm]{Example}

\crefname{lem}{Lemma}{Lemmas}
\crefname{thm}{Theorem}{Theorems}
\crefname{defn}{Definition}{Definitions}
\crefname{prop}{Proposition}{Propositions}
\crefname{rmk}{Remark}{Remarks}
\crefname{cor}{Corollary}{Corollaries}
\crefname{ex}{Example}{Examples}
\crefname{notation}{Notation}{Notations}
\crefname{convention}{Convention}{Conventions}
\crefname{enumi}{}{}
\crefname{indhyp}{Inductive Hypothesis}{Inductive Hypotheses}

\newlist{abc}{enumerate}{1}
\setlist[abc]{label=(\alph*)}

\newlist{rome}{enumerate}{7}
\setlist[rome]{label=(\roman*)}

\newcommand{\cat}{\mathrm{Cat}}
\newcommand{\gpd}{\mathrm{Gpd}}
\newcommand{\set}{\mathrm{Set}}
\newcommand{\sset}{s\set}
\newcommand{\topo}{\mathrm{Top}}
\newcommand{\ssetp}{\sset_*}
\newcommand{\ssetpn}{{\ssetp}_{[0,n]}} 
\newcommand{\topn}{\topo_{[0,n]}} 
\newcommand{\Gsset}{\Gamma\ssetp}
\newcommand{\vsGsset}{\mathrm{v.s.}\Gsset} 
\newcommand{\Gssetn}{\Gamma\ssetpn} 
\newcommand{\vsGssetn}{\mathrm{v.s.}\Gssetn}
\newcommand{\Tam}{\mathrm{Tam}}
\newcommand{\GTam}{\mathrm{GTam}}
\newcommand{\vsGammaGTam}{\mathrm{v.s.}\Gamma\GTam}
\newcommand{\PicGTam}{\mathrm{PicTam}}

\newcommand{\spectra}{\mathrm{Sp}}
\newcommand{\cospectra}{\Omega\spectra_{\geq 0}} 
\newcommand{\spectran}{\spectra_{[0,n]}} 
\newcommand{\ospectran}{\Omega\spectra_{[0,n]}}
\newcommand{\Sym}{\mathrm{SymMonCat}}
\newcommand{\Nhom}{\mathrm{NHom}}
\newcommand{\GammaCat}{[\Gamma,\cat]}
\newcommand{\GammaGpd}{[\Gamma,\gpd]}
\newcommand{\Pic}{\mathrm{Pic}}
\newcommand{\GammaGTam}{\left[\Gamma,\GTam^n\right]}
\newcommand{\GammasSet}{\left[\Gamma,\sset\right]}

\newcommand{\C}{\mathcal{C}}
\newcommand{\clW}{\mathcal{W}}

\newcommand{\R}{\mathcal{R}}

\newcommand{\Hp}{\mathcal{H}}
\newcommand{\X}{\mathcal{X}} 
\newcommand{\Y}{\mathcal{Y}}
\newcommand{\XX}{\mathcal{X}} 
\newcommand{\YY}{\mathcal{Y}} 

\newcommand{\KP}{\mathcal{K}P}
\newcommand{\NP}{N\mathcal{P}}
\newcommand{\ord}[1]{\langle {#1} \rangle}
\newcommand{\und}[1]{\underline{#1}}

\newcommand{\id}{\mathrm{id}}
\newcommand{\Dop}{\Delta^{\mathrm{op}}}
\newcommand{\Dnop}{(\Delta^{\op})^{n}}
  \newcommand\Dsop[1]{\ifthenelse{\equal{#1}{}}{\Delta^{\op}}{(\Delta^{\op})^{#1}}}
\renewcommand{\S}{\mathds{S}}
\newcommand{\cate}{\tau_1}
\newcommand{\deltatop}[1]{\Delta_{\mathrm{top}}[#1]}

\newcommand{\multising}[2]{\mathcal{S}^{#1}(#2)}

\DeclareMathOperator{\op}{op}
\DeclareMathOperator{\colim}{colim}
\DeclareMathOperator{\Ho}{Ho}
\DeclareMathOperator{\HOM}{Map} 
\DeclareMathOperator{\Hom}{Hom}
\DeclareMathOperator{\Sing}{Sing}
\DeclareMathOperator{\pr}{pr}
\DeclareMathOperator{\diag}{diag}

\newenvironment{tz}{\begin{center}\begin{tikzpicture}[scale=1]}{\end{tikzpicture}\end{center}}
\tikzstyle{d}=[double distance=.3ex]

\newcommand{\tiund}[1]{{\times}_{#1}}
\newcommand{\pro}[3]{#1\tiund{#2}\overset{#3}{\cdots}\tiund{#2}#1}
\newcommand{\funcat}[2]{[\Dsop{#1},#2]}
\newcommand{\p}[1]{p^{(#1)}}
\newcommand{\seqc}[3]{{#1}_{#2},\ldots,{#1}_{#3}}

\newcommand{\mi}{\text{-}}

\thanks{}

\begin{document}
\title{Stable homotopy hypothesis in the Tamsamani model}
\begin{abstract} We prove that symmetric monoidal weak $n$-groupoids in the Tamsamani model provide a model for stable $n$-types. Moreover, we recover the classical statement that Picard categories model stable $1$-types.
\end{abstract}

\author[L.\ Moser]{Lyne Moser}
\address{UPHESS BMI FSV, Ecole Polytechnique Fédérale de Lausanne, Station 8, 1015 Lausanne,
Switzerland}
\email{lyne.moser@epfl.ch}

\author[V.\ Ozornova]{Viktoriya Ozornova}
\address{Fakult\"at f\"ur Mathematik, Ruhr-Universit\"at Bochum, D-44780 Bochum, Germany}
\email{viktoriya.ozornova@rub.de}

\author[S.\ Paoli]{Simona Paoli}
\address{School of Mathematics and Actuarial Science, University of Leicester, Leicester LE17RH, United Kingdom}
\email{sp424@le.ac.uk}

\author[M.\ Sarazola]{Maru Sarazola}
\address{Department of Mathematics, Cornell University, Ithaca NY, 14853, USA}
\email{mes462@cornell.edu}

\author[P.\ Verdugo]{Paula Verdugo}
\address{Department of Mathematics, Macquarie University, NSW 2109, Australia}
\email{paula.verdugo@hdr.mq.edu.au}
\maketitle

\tableofcontents

\section{Introduction}
Higher categories are designed to capture different flavours and levels of compositions. From this perspective, it is somewhat surprising that Grothendieck suggested already in the 1970s that spaces should be the same as $\infty$-groupoids. This connection becomes slightly less mysterious if we think about the points of a topological space as objects, paths between them as morphisms, homotopies between paths as $2$-morphisms, homotopies between homotopies between paths as $3$-morphisms, and so on, where the morphisms in each of the higher levels can be composed by pasting the homotopies together. Grothendieck's paradigm was made precise in various ways, and provided a fruitful connection between higher category theory and topology. Versions of this statement are usually summarized under the name \emph{homotopy hypothesis}. More refined versions of this statement make it possible to identify intermediate levels, namely $n$-groupoids on one side with $n$-types on the other side, i.e.,\ with spaces whose homotopy groups are concentrated in dimensions $\leq n$.

A word of warning is in order at this point. We can fill the notion of $n$-groupoids with different meanings, and it turns out not all of them work equally well for the purpose of the homotopy hypothesis. Most notably, the easiest and least ambiguous notion of strict $n$-groupoid is not suitable to model $n$-types for any $n\geq 3$, as was unexpectedly shown by Simpson \cite{SimpsonCounterexample}. To make the homotopy hypothesis true, we are bound to work with weak higher categories, and there are many different models available to encode such structures. At first glance, it is the nicest case to have a fully algebraic model, i.e.,\ one that can be described by a finite amount of data subject to a finite list of axioms, as done e.g.\ in the case of bicategories. However, this approach gets out of hand very quickly: while tricategories are still in use as, for example, in \cite{GordonPowerStreet}, \cite{GurskiTricat}, \cite{CegarraHeredia}, \cite{CarquevilleMeusburgerSchaumann}, and the definition of tetracategories can still be made explicit in principle \cite{TrimbleTetra}, any higher structure seems out of reach.

This leads to the conclusion that we need another way to encode higher-categorical information. There are various approaches for doing so, nicely summarized in \cite{Leinster}. The common idea of all of them is to provide a framework which automatically captures all the higher coherence data and axioms for them. We concentrate in this paper on one of the Segal-type models, which are treated in great detail in \cite{Paoli-2019}. The basic idea is as follows. Ordinary categories can be considered as a full subcategory of simplicial sets via the nerve functor. Its essential image can be characterized by the Segal condition, saying for a simplicial set $X$ that a certain map $X_n \to X_1 \times_{X_0} X_1 \times_{X_0} \ldots \times_{X_0} X_1$ built out of simplicial structure maps has to be an  isomorphism. This idea was generalized by Tamsamani to make $n$-fold simplicial sets satisfying certain iterated weak Segal conditions into the model of weak $n$-categories, which we today call \emph{Tamsamani $n$-categories}. With this notion of higher categories, Tamsamani was able to prove the homotopy hypothesis, showing that the homotopy category of Tamsamani $n$-groupoids is equivalent to the homotopy category of $n$-types \cite{Tamsamani}.

The category of spaces is quite complicated and mysterious. One ubiquitous way to access it is to linearize it to a degree, by working with the stable homotopy category instead. The \emph{stable homotopy hypothesis} asserts that at least certain stable homotopy types have a categorical incarnation, in the form of symmetric monoidal higher categories. First versions of this principle for the case of symmetric monoidal groupoids can be found in \cite{Sinh}, \cite{DrinfeldInfDim}, \cite{BoyarchenkoNotes}, \cite{Patel}, \cite{JohnsonOsorno}; we refer the reader to \cite{GurskiJohnsonOsornoNotices} for an accessible survey and historical account. It is also closely related to the Baez--Dolan stabilization hypothesis, which was proven in $\infty$-categorical terms for the Tamsamani model in \cite{GepnerHaugseng}. 

The stable homotopy types we aim to model are stable $n$-types, i.e.,\ those whose homotopy groups are concentrated in the interval $[0,n]$. On the categorical side, we encounter the same problem as in the -- unstable -- homotopy hypothesis, only that it starts earlier. Symmetric monoidal categories can be viewed as bicategories with one object and extra symmetry and thus have roughly the same complexity as bicategories. One dimension further, a symmetric monoidal bicategory is a special case of a tricategory with an extra datum and is already quite awkward to write down explicitly. In particular, this makes the homotopy hypothesis for stable $2$-types \cite{GurskiJohnsonOsornoStephan}, \cite{GurskiJohnsonOsornoSHH} very complicated. Another complication of a possibly more fundamental nature is the fact that tricategories cannot be strictified in general \cite{GordonPowerStreet}, as opposed to bicategories which are always equivalent to a strict $2$-category. This strictification is a result on which previous work for lower levels heavily relies on.

This makes it necessary to work with a different model for higher categories again. A well-known way to encode symmetric monoidal structures in different contexts goes back to Segal \cite{SegalCat} and uses the category $\Gamma$ of finite pointed sets and basepoint preserving maps between them, cf.\ e.g.\ \cite{LeinsterHoAlgebra}, \cite{LurieHTT}. We use this approach to define a version of symmetric monoidal weak $n$-groupoids which we call \emph{Picard--Tamsamani $n$-categories}. With this notion, we are able to prove the following main result.

\begin{theorem*}[\cref{StableHH}]
The homotopy category of Picard--Tamsamani $n$-categories is equivalent to the homotopy category of stable $n$-types.
\end{theorem*}
Since there were already versions of the stable homotopy hypothesis available for $n\leq 2$, it is reasonable to ask for a comparison of the theorem above to the existing statements. We address this question in the case of stable $1$-types. It is well-known that Picard categories model stable $1$-types, although the statement seemed to be folklore for a long time, and there is a well-developed algebraic theory for Picard categories (sometimes under the name \lq\lq{}abelian $2$-groups\rq\rq{}) \cite{JibladzePirashvili}, \cite{Pirashvili}. To connect our results to the known statements, we observe that Segal's $K$-theory construction \cite{SegalCat} is closely related, when restricted to Picard categories, to the Lack--Paoli $2$-nerve from \cite{LackPaoli}.

\begin{theorem*}[\cref{KvsNerveThm}] For any Picard category $P$, there is a natural equivalence $\KP\to \NP$ between the underlying simplicial object of Segal's $K$-theory construction on $P$ and the Lack--Paoli $2$-nerve of the bicategory $\mathcal P$ with one object associated to $P$. 
\end{theorem*}

Finally, we follow the lines of Lack--Paoli's discussion of the $2$-nerve to recover a version of the stable homotopy hypothesis for Picard categories. Since the first version of this article appeared, a different proof of a version of this result appeared in \cite{SharmaPicard}, based on the previous work \cite{SharmaSymMonCat} of the author.

\begin{cor*}[\cref{CorollaryPicard}]
The equivalence between the homotopy category of Picard--Tamsamani $1$-categories and the homotopy category of stable $1$-types specializes to an equivalence between the homotopy categories of Picard categories and stable $1$-types.
\end{cor*}

The structure of this paper is as follows. In \cref{Sec:TamCat}, we recall the definitions and properties of Tamsamani $n$-categories and how they relate to unstable $n$-types. In \cref{Sec:Spectra}, we give an overview of Bousfield--Friedlander spectra which we use to model stable types, as well as of their variant of Segal's theory of $\Gamma$-spaces. In \cref{Sec:SHH}, we prove our main theorem. In the last two sections, we relate this theorem to previous work, proving the two mentioned results for Picard categories. 

\subsection*{Acknowledgements} 
This project arose as a part of the Women in Topology III program. We would like to thank the organizers for making this collaboration possible. We would also like to thank the Hausdorff Research Institute for Mathematics, NSF-DMS 1901795, NSF-HRD 1500481 - AWM ADVANCE grant and Foundation Compositio Mathematica for supporting this program. Our special thanks go to Ang\'{e}lica Osorno, to whom we owe the original impetus of this project, and who supported us throughout the work on it. The fifth-named author acknowledges the financial support of an international Macquarie University Research Excellence Scholarship. 

\section{Tamsamani categories}
\label{Sec:TamCat}
In this section, we recall Tamsamani's definition of weak $n$-category. This is the model of weak higher categories we are going to work with throughout the paper. The definition is inductive in nature and requires several preparations. Along the way, we collect facts about Tamsamani categories needed in this paper. In particular, we will be interested in the notion of equivalences of Tamsamani categories, which generalizes the notion of equivalences of categories. We do not make any claim of originality in this part. For a more thorough treatment, we refer the reader e.g.\ to \cite{Tamsamani}, \cite{Simpson}, \cite{Paoli-2019}. 

\subsection{Notational conventions}
Tamsamani categories are multi-simplicial objects, and we first need to fix some notation for these. As usual, $\Delta$ will denote the simplex category, whose objects are non-empty, finite, totally ordered sets $[n]=\{0,\dots,n\}$, and whose arrows are order preserving maps between them. To avoid confusion with several copies of the simplex category, we denote by $\Delta[n]$ the simplicial set represented by $[n]$.

For each non-negative integer $n$, we write $d^i\colon [n-1]\to [n]$ for generating face maps and $s^i\colon [n+1]\to [n]$ for generating degeneracy maps for all $0\leq i\leq n$.

We write $\Delta^n$ for the category $\Delta\times\cdots\times \Delta$ of $n$ copies of the simplex category. Its objects, tuples $([m_1], \cdots,  [m_n])$, will be denoted by $M=(m_1,\dots,m_n)$. Its arrows, $\alpha\colon M\to M^\prime$, are $n$-tuples $\alpha=(\alpha_1,\dots,\alpha_n)$ with $\alpha_k\colon [m_k]\to [m^\prime_k]$ in $\Delta$ for every $1\leq k\leq n$.

Let $\Dnop$ denote the product of $n$ copies of $\Dop$. Given a category $\C$, the functor category $[\Dsop{n},\C]$ is called the category of $n$-fold simplicial objects in $\C$ (or multi-simplicial objects for short).

If $X\in [\Dsop{n},\C]$ and $\underline{k}=(k_1,\ldots,k_n)\in \Dnop$, we shall write $X_{k_1, \ldots, k_n}$ for the value of $X$ at this object. 

When we set $\C$ to be the category $\set$ of sets, the category $[\Dsop{n}, \set]$ is the category of $n$-fold simplicial sets.

In order to efficiently handle the relationship between Tamsamani models for different dimensions, we introduce a special notational convention as follows.

Let $d\colon\set\to[\Dop, \set]$ be the discrete simplicial set functor. Clearly $d$ is fully faithful, and therefore so is the induced functor
\begin{equation*}
  d_*\colon [\Dsop{n-1}, \set]\to \funcat{n-1}{\funcat{}{\set}}
\end{equation*}
with $(d_* X)_{\seqc{k}{1}{n-1}}=d X_{\seqc{k}{1}{n-1}}$ for all $(\seqc{k}{1}{n-1})\in\Dsop{n-1}$. There is an isomorphism 
\begin{equation*}
  \funcat{n-1}{\funcat{}{\set}}\cong \funcat{n}{\set}
\end{equation*}
associating to $Y\in \funcat{n-1}{\funcat{}{\set}}$ the $n$-fold simplicial set taking $(\seqc{k}{1}{n})\in \Dsop{n}$ to $(Y_{k_1,\ldots, k_{n-1}})_{k_n}$. Hence the composite functor
\begin{equation}\label{eq-multsimpsets-1R}
  \funcat{n-1}{\set}\xrightarrow{d_*}\funcat{n-1}{\funcat{}{\set}}\cong \funcat{n}{\set}
\end{equation}
is fully faithful. This property justifies the following convention.

\begin{convention}\label{conv-multsimpsets-1R}
We identify $\funcat{n-1}{\set}$ with its essential image in $\funcat{n}{\set}$ under the functor \eqref{eq-multsimpsets-1R}. In other words, we identify $(n-1)$-fold simplicial sets with those $n$-fold simplicial sets $X$ for which the simplicial set $X_{\seqc{k}{1}{n-1}}$ is discrete for all $(\seqc{k}{1}{n-1})\in\Dsop{n-1}$. 
\end{convention}

To encode composition in different directions in Tamsamani's model, one uses various Segal conditions. We recall what these conditions are and settle the notation for them. 

\begin{defn}\label{def-seg-map}
Let ${X\in\funcat{}{\C}}$ be a simplicial object in a category $\C$ with pullbacks. For each ${1\leq j\leq k}$ and $k\geq 2$, let ${\nu_j\colon X_k\to X_1}$ be induced by the map  $\nu^j\colon [1]\to[k]$ in $\Delta$ sending $0$ to ${j-1}$ and $1$ to $j$. Then the following diagram commutes:
\begin{equation*}
\begin{tikzcd}[>=stealth]
&&& X_k\ar[dll, "\nu_1"] \ar[d, "\nu_2"] \ar[drrrr, "\nu_k"] &&&&& \\
&X_1\ar[ld, "d_1" swap]\ar[rd, "d_0"]&&X_1\ar[ld, "d_1" swap]\ar[rd, "d_0"]&&\ldots &&X_1\ar[ld, "d_1" swap]\ar[rd, "d_0"]&\\
X_0 && X_0 && \ldots &&X_0&& X_0
\end{tikzcd}
\end{equation*}

If  ${\pro{X_1}{X_0}{k}}$ denotes the limit of the lower part of the diagram above, the \emph{$k$-th Segal map of $X$} is the unique map
$$
S_k\colon X_{k}~\to~\pro{X_{1}}{X_{0}}{k}
$$
such that ${\pr_j\circ S_k=\nu_{j}}$ where
${\pr_{j}}$ is the $j$-{th} projection.
\end{defn}

\subsection[The functor $\p{r}$ ]{The functor $\pmb{\p{r}}$} \label{Funp}
We recall some basic constructions from \cite[\textsection 2.2.1]{Paoli-2019}. 
Let $\cat$ be the category of small categories. There is a fully faithful nerve functor $N\colon \cat \to \funcat{}{\set}$ with a left adjoint $\cate$, often called the \emph{homotopy category} construction.

\begin{defn}\label{def-multsimpsets-1R-A}
Let $p\colon\funcat{}{\set} \to \set$ be obtained by applying $\cate$ to $X\in \funcat{}{\set}$ and then taking the set of isomorphism classes of the category $\cate X$.
\end{defn}

\noindent We can extend the functor $p$ as follows.

\begin{defn}\label{def-multsimpsets-1R}
Define inductively
\begin{equation*}
  p_n \colon\funcat{n}{\set} \to \set
\end{equation*}
by setting $p_1=p$ and, given $p_{n-1}$, by letting $p_n$ be the composite
\begin{equation}\label{eq-multsimpsets-2R}
\begin{array}{ll}
  \funcat{n}{\set} &\cong \funcat{n-1}{\funcat{}{\set}}\xrightarrow{p_*}\funcat{n-1}{\set} \xrightarrow{p_{n-1}}\set\;,\\
\end{array}
\end{equation}
where the isomorphism on the left-hand side of \eqref{eq-multsimpsets-2R} associates to $X\in\funcat{n}{\set}$ the object of $\funcat{n-1}{\funcat{}{\set}}$ taking $(k_1, \ldots, k_{n-1})\in \Dsop{n-1}$ to the simplicial set $X_{k_1, \ldots, k_{n-1}}$.
\end{defn}

\begin{convention}\label{conv-multsimpsets-2R}
Under the embedding
\begin{equation*}
  \funcat{n-1}{\set}\hookrightarrow \funcat{n}{\set}
\end{equation*}
of \cref{conv-multsimpsets-1R}, we see from \cref{def-multsimpsets-1R} that the following diagram commutes:
\begin{equation*}
\begin{tikzcd}
\funcat{n-1}{\set} \ar[rr, hook] \ar[drr, "{p_{n-1}}" swap] && \funcat{n}{\set} \ar[d, "{p_n}"] \\
&& \set
\end{tikzcd}
\end{equation*}
Consequently, no ambiguity can arise by dropping subscripts in the definition of $p_n$ and simply writing
\begin{equation}\label{eq-multsimpsets-3R}
   p\colon  \funcat{n}{\set} \to \set\;.
\end{equation}
\end{convention}

We now extend the definition of $p$ as follows.

\begin{defn}\label{def-multsimpsets-2R}
Let $p\colon\funcat{n}{\set}\to \set$ be as in \eqref{eq-multsimpsets-3R}. For each $0\leq r\leq n-1$, define
\begin{equation*}
  \p{r}\colon\funcat{n}{\set}\to \funcat{r}{\set}
\end{equation*}
by $\p{0}=p$ and, for $1\leq r\leq n-1$, $(k_1, \ldots, k_r)\in\Dsop{r}$, and $X\in\funcat{n}{\set}$, we set
\begin{eqnarray*}
(\p{r}X)_{k_1, \ldots, k_r} &=& p (X_{k_1, \ldots, k_r})\;, 
\end{eqnarray*}
and similarly for the morphisms.
\end{defn}

\begin{rmk}\label{rem-multsimpsets-1R}
  Using  \cref{conv-multsimpsets-1R}, we have a commuting diagram for each $0\leq r \leq n-1$
\begin{equation*}
\begin{tikzcd}
 \funcat{n-1}{\set} \ar[rr, hook] \ar[drr, "{\p{r}}" swap] && \funcat{n}{\set} \ar[d, "{\p{r}}"] \\
 && \funcat{r}{\set}
\end{tikzcd}
\end{equation*}
In the definition of $\p{r}$ we can therefore drop explicit mention of the source dimension $n$.
\end{rmk}

The following lemma establishes some elementary properties of the functor $\p{r}$.

\begin{lem}[{\cite[Lemma 2.2.7]{Paoli-2019}}]\label{lem-multsimpsets-2R} 
For each $0\leq r \leq n-1$, let $\p{r}$ be as in \cref{def-multsimpsets-2R}. Then 
\begin{enumerate}[label=\alph*)]
  \item For each $[s]\in\Dop$,
\begin{equation*}
  (\p{r}X)_s = \p{r-1}X_s. 
\end{equation*}

  \item For each $0\leq r\leq n-1$, the functor $\p{r}\colon\funcat{n}{\set}\to \funcat{r}{\set}$ factors as
  \begin{equation*}
    \funcat{n}{\set}\xrightarrow{\p{n-1}}\funcat{n-1}{\set} \xrightarrow{\p{n-2}} \funcat{n-2}{\set} \cdots \xrightarrow{\p{r}}\funcat{r}{\set}
  \end{equation*}
\end{enumerate}
\end{lem}

\subsection{Defining Tamsamani $n$-categories}

The aim of this subsection is to actually define Tamsamani $n$-categories. The definition was first given in \cite{Tamsamani}. Our equivalent presentation follows the lines of \cite{Paoli-2019}. Note that the definition is inductive and does not make sense before we have proven that Tamsamani $(n-1)$-categories and equivalences of such satisfy certain properties. We dedicate the rest of this subsection to checking these properties. 

\begin{ClosureProp}\label{subs-clos-prop}
The Tamsamani model is a full subcategory of $\funcat{n}{\set}$ satisfying a number of closure properties, which we discuss below for future reference.

Let $\C \subset \funcat{n}{\set}$ be a full subcategory of $n$-fold simplicial sets. We consider the following closure properties
 of $\C$:
\begin{enumerate}[label=\textbf{C\arabic*:}, ref=\textbf{C\arabic*}, start=0]
 
  \item \label{C0} The subcategory $\C$ contains the terminal object. 
  \item \label{C1} Repletion under isomorphisms; that is, if $A\cong B$ in $\funcat{n}{\set}$ and $A\in \C$, then $B\in \C$.

  \item \label{C2} Closure under finite products.

  \item \label{C3} Closure under small coproducts.

  \item \label{C4} If the small coproduct $\underset{i}{\amalg} A_i$ is in $\C$, then each $A_i\in\C$.
\end{enumerate}
\end{ClosureProp}

Once the closure properties above hold, we get the following properties for free. These properties are precisely what we need in order to make sense of Segal-type conditions for Tamsamani $n$-categories. 
\begin{lem}[{\cite[Lemma 2.2.8]{Paoli-2019}}]\label{lem-multsimpsets-3R} 
Let $\C$ be a full subcategory of $\funcat{n}{\set}$ satisfying the closure properties \cref{C0,C1,C2,C3,C4}. Then: 

\begin{enumerate}
  \item Every discrete $n$-fold simplicial set is an object of $\C$.

  \item If $X$ is a discrete $n$-fold simplicial set and $f\colon E\to X$ is a morphism in $\C$, then for each $x\in X$ the
      fiber $E_x$ of $f$ at $x$ is in $\C$.

  \item If $A\xrightarrow{f} X  \xleftarrow{g} B$ is a diagram in $\C$ with $X$ discrete, then $A\tiund{X}B\in\C$.
\end{enumerate}
\end{lem}

We define the category $\Tam^n\subset\funcat{n}{\set}$ and $n$-equivalences by induction on $n$. When $n=0$, we have $ \Tam^{0}=\set$ and $0$-equivalences are bijections. When $n=1$, $ \Tam^{1}=\cat\hookrightarrow \funcat{}{\set}$ and 1-equivalences are equivalences of categories.

\begin{indhyp} \label{Inductive hypothesis}
Suppose, inductively, that we defined for each $1\leq k\leq n-1$ a full subcategory
\begin{equation*}
   \Tam^{k}\subset \funcat{k-1}{\cat}\subset \funcat{k}{\set}
\end{equation*}
containing the terminal object and a class $\clW_k$ of maps in $ \Tam^{k}$ (called \emph{$k$-equivalences}) such that

\begin{enumerate}[label=\textbf{I\arabic*:}, ref=\textbf{I\arabic*}]
  \item\label{I1} $ \Tam^{k}$ satisfies the closure properties \cref{C0,C1,C2,C3,C4}, 

  \item\label{I2} The functor $\p{k-1}\colon\funcat{k}{\set}\to\funcat{k-1}{\set}$ of \cref{def-multsimpsets-2R} restricts to a functor $\p{k-1}\colon \Tam^{k}\to \Tam^{k-1}$ which sends $k$-equivalences to $(k-1)$-equivalences. 

  \item\label{I3} $\clW_k$ is closed under composition with isomorphisms, it is closed under finite products and small colimits and, if the small colimit $\underset{i}{\amalg}f_i$ of maps in $\Tam_k$ is in $ \clW_{k}$, then each $f_i\in \clW_{k}$. 
\end{enumerate}
\end{indhyp}

\begin{defn} \label{DefTamsamani}
An object $\X$ of $\funcat{n-1}{\cat}\subset\funcat{n}{\set}$ is a \emph{Tamsamani $n$-category} if
\begin{enumerate}[label=\alph*)]
  \item $\X_0$ is discrete. 

  \item $\X_s\in \Tam^{n-1}$ for all $s>0$.

  \item For all $s\geq 2$, the Segal maps $\X_s\to\pro{\X_1}{\X_0}{s}$ are $(n-1)$-equivalences. 
\end{enumerate}
\end{defn}

Note that, since $ \Tam^{n-1}$ satisfies the closure properties \cref{C0,C1,C2,C3,C4}, by \cref{lem-multsimpsets-3R}, the iterated pullback $\pro{\X_1}{\X_0}{s}$ is in $\Tam^{n-1}$, so that the last condition in the definition actually makes sense. The morphisms of Tamsamani $n$-categories are just morphisms between underlying multi-simplicial sets. 

To satisfy the inductive step, we check first that the functor $\p{n-1}\colon\funcat{n}{\set}\to\funcat{n-1}{\set}$ of \cref{def-multsimpsets-2R} restricts to a functor $\p{n-1}\colon\Tam^n\to \Tam^{n-1}$. Since $\p{n-2}$ preserves pullbacks over discrete objects (as the same is true for $p$), we have
      \begin{equation*}
        \p{n-1}(\pro{\X_1}{\X_0}{s}) \cong \pro{\p{n-1}\X_1}{\p{n-1}\X_0}{s}
      \end{equation*}
      and, by hypothesis, $\p{n-2}$ sends $(n-1)$-equivalences to $(n-2)$-equivalences. Therefore the Segal maps, being $(n-1)$-equivalences, give rise to $(n-2)$-equivalences
      \begin{equation*}
        \p{n-2}\X_s\to \pro{\p{n-2}\X_1}{\p{n-2}\X_0}{s}\;.
      \end{equation*}
      This shows that $\p{n-1}\X\in \Tam^{n-1}$.

Next, we define the $(n-1)$-categories of morphisms in a Tamsamani $n$-category. This will allow us to define $n$-equivalences, roughly speaking, as functors which are fully faithful and essentially surjective.  

\begin{defn} Let $\X$ be a Tamsamani $n$-category. 
Given $a,b\in \X_0$, the \emph{$(n-1)$-category of morphisms} $\X(a,b)$ is defined to be the fiber at $(a,b)$ of the map $(d_1,d_0)\colon \X_1\to\X_0\times \X_0$.
\end{defn}

Since, by inductive hypothesis, $ \Tam^{n-1}$ satisfies \cref{C0,C1,C2,C3,C4}, by \cref{lem-multsimpsets-3R} we have $\X(a,b)\in \Tam^{n-1}$. One should think of $\X(a,b)\in \Tam^{n-1}$ as a $\hom\mi(n-1)$-category. 

\begin{defn}
We define a map $f\colon \X\to \Y$ in $\Tam^n$ to be an \emph{$n$-equivalence} if
\begin{rome}
  \item For all $a,b\in \X_0$, the induced morphism $f(a,b)\colon \X(a,b)\to \Y(fa,fb)$ is an $(n-1)$-equivalence. 

  \item $\p{n-1} f$ is an $(n-1)$-equivalence.
\end{rome}
\end{defn}

To complete the inductive step in the definition of $\Tam^n$, we need to check that $\Tam^n$ satisfies the inductive hypotheses \cref{I1,I2,I3} at step $n$. 

\begin{lem}\label{CompletingInductiveHypothesis}
The subcategory $\Tam^n\subset\funcat{n-1}{\cat}\subset \funcat{n}{\set}$ satisfies the inductive hypotheses \cref{I1,I2,I3} at step $n$.
\end{lem}

\begin{proof}
The property \cref{I2} has already been checked above, and the fact that $\p{n-1}$ sends $n$-equivalences to $(n-1)$-equivalences is part of the definition of $n$-equivalence in $\Tam^n$.

Finally, \cref{I1} and \cref{I3} follow by \cite[Prop.\ 6.1.6]{Paoli-2019}, taking $\C_{n-1}= \Tam^{n-1}$ with $\clW_{n-1}$ the $(n-1)$-equivalences and $\C_{n}=\Tam^n$ with $\clW_{n}$ the $n$-equivalences. 
\end{proof}

To illustrate the concept, we spell out what it means to be a Tamsamani $2$-category.

\begin{ex}\label{example 1} 
From the definition, $\X\in \Tam^{2}$ consists of a simplicial object $\X\in\funcat{}{\cat}$ such that $\X_0$ is discrete and the Segal maps $\X_s\to\pro{\X_1}{\X_0}{s}$ are equivalences of categories. The functor $\p{1}\colon \Tam^{2}\to\cat$ associates to $\X\in \Tam^{2}$ the simplicial set taking $[k]\in\Dop$ to $p(\X_k)$; this simplicial set is the nerve of a category since, for each $k\geq 2$,
\begin{equation*}
p \X_k\cong p(\pro{\X_1}{\X_0}{k})\cong \pro{p(\X_1)}{p(\X_0)}{k}\;.
\end{equation*}

We will see in \cref{NP} that any bicategory provides us with an example of a Tamsamani $2$-category. Actually, by \cite{LackPaoli}, any Tamsamani $2$-category arises in this way up to an appropriate equivalence.
\end{ex}

In general, an $n$-equivalence of Tamsamani $n$-categories is not a levelwise $(n-1)$-equivalence, e.g.\ an equivalence of categories is not always a levelwise bijection on the nerves. However, it is helpful to identify the additional condition making an $n$-equivalence into a levelwise $(n-1)$-equivalence.

\begin{lem}[{\cite[Lemma 7.1.3]{Paoli-2019}}] \label{leveleqvsneq}
A map $f\colon \X\to \Y$ in $\Tam^n$ is a levelwise $(n-1)$-equivalence in $\Tam^{n-1}$ if and only if it is an $n$-equivalence and $pf_0\colon p\X_0\to p\Y_0$ is a bijection. 
\end{lem}

\begin{rmk}\label{equiv_Tam_cats_woo}
Note that \cite[Lemma 7.1.3]{Paoli-2019} is actually formulated in more generality. In our case, $pf_0\colon p\X_0\to p\Y_0$ can be combined with the fact that both $\X_0$ and $\Y_0$ are discrete, so that an equivalent formulation would be to say that $f_0$ gives a bijection on underlying sets of $\X_0$ and $\Y_0$. In particular, if $\X,\Y\in\Tam^n$ are Tamsamani $n$-categories with one object, then a map $f\colon \X\to \Y$ is an $n$-equivalence if and only if it is a levelwise $(n-1)$-equivalence. 
\end{rmk}

The homotopy hypothesis requires a notion of $n$-groupoids in the given model of weak $n$-categories, and we introduce the corresponding notion next.

\begin{defn}\label{Definition GTa}
The full subcategory $ \GTam^{n}\subset \Tam^n$ of Tamsamani $n$-groupoids is defined inductively as follows.
For $n=0$, let $\GTam^0=\Tam^0=\set$ be the category of sets. For $n=1$, let $\GTam^{1}$ be the full subcategory $\gpd\subset \cat$ of groupoids. Suppose, inductively, that we defined the full subcategory $ \GTam^{n-1}\subset  \Tam^{n-1}$. We say that a Tamsamani $n$-category $\X$ is in $ \GTam^{n}$ if it satisfies
\begin{itemize}
  \item [i)] $\X_k\in \GTam^{n-1}$ for all $k\geq 0$. 

  \item [ii)] $\p{n-1}\X\in \GTam^{n-1}$.
\end{itemize}
\end{defn}

It follows immediately from this definition that the embedding $\Tam^n\hookrightarrow \funcat{n-1}{\cat}$ restricts to an embedding $ \GTam^{n}\hookrightarrow\funcat{n-1}{\gpd}$. If $\X\in \GTam^{n}$, then by definition $\X_1\in \GTam^{n-1}$, thus for each $a,b\in \X_0$, we see that $\X(a,b)\in \GTam^{n-1}$. 

\subsection{The functor $\pi_0$ for Tamsamani $n$-groupoids} \label{PropertiesTamsamani}
We define the functor $\pi_0$ as the restriction of the functor $p$, defined in \cref{Funp}, to Tamsamani $n$-groupoids. We will collect some basic properties of this functor which will be needed to prove the stable homotopy hypothesis. 
The following definition is just a relabelling and could have been done for all Tamsamani $n$-categories. However, in the case of groupoids, it coincides with the $\pi_0$ of the corresponding classifying space as we show in \cref{pi0_GTam_compatibility}.

\begin{defn} 
For each $n$, we define the functor 
\[
\pi_0\colon\GTam^n\to\set
\]
to be the restriction of the functor $p=p^{(0)}\colon\Tam^n\to\set$ of \cref{Inductive hypothesis}, which was in turn the restriction of the functor $p=p^{(0)}\colon \funcat{n}{\set}\to \set$ of \cref{def-multsimpsets-2R} to Tamsamani $n$-categories.
\end{defn}

To make this construction more explicit, let explain in words what $p\X$ is for a Tamsamani $n$-category $\X$. Similarly as in the case of bicategories, one can define a notion of internal equivalence in a Tamsamani $n$-category. Then $p$ is the functor that assigns to each Tamsamani $n$-category $\X\in\Tam^n$ the set of equivalence classes of objects of $\X$ with respect to internal equivalences. For a more detailed discussion, we refer the reader to \cite[\textsection 3]{Tamsamani}.

\begin{ex}
For $n=1$, recall that the equivalence classes are precisely the isomorphism classes of objects, so that $p\X$ for a Tamsamani $1$-category $\X$ is 
the set of isomorphism classes of objects in $\X$.

For $n=2$, it is known \cite{LackPaoli} that every Tamsamani $2$-category comes from a bicategory, as we will recall in \cref{NP} in detail. For a Tamsamani $2$-category $\X$, the set $p\X$ gives precisely the objects modulo equivalences in $\X$. Indeed, after applying $p^{(1)}$ to $\X$, we obtain the category whose morphisms are isomorphism classes of old $1$-morphisms, and isomorphisms in this category are precisely the equivalences in the original bicategory. 
\end{ex} 

\begin{lem}
The functor $p\colon \Tam^n\to \set$ preserves finite products, and hence so does the functor $\pi_0\colon \GTam^n\to \set$.
\end{lem} 

\begin{proof} 
We know by \cite[Lemma 4.1.4]{Paoli-2019} that the functor $p\colon\cat\to\set$ preserves finite products; then, so does each $p^{(k)}$, and in particular, $p=p^{(0)}$ and therefore also $\pi_0$. 
\end{proof}

The following result is easy to derive from the definitions and was first obtained in \cite[Prop.\ 5.4]{Tamsamani}.

\begin{prop}\label{n_equiv_pi0}
Let $f\colon \X\to \Y$ be a map in $\GTam^n$. If $f$ is an $n$-equivalence, then $\pi_0f\colon\pi_0\X\to\pi_0\Y$ is a bijection.
\end{prop}

\subsection{Homotopy hypothesis in the Tamsamani model}

Our next goal is to review Tamsamani's version of the (unstable) homotopy hypothesis. This subsection is dedicated to defining the Poincar\'{e} $n$-groupoid of a topological space, which will constitute one half of the equivalence between truncated homotopy types and higher groupoids. We need to go into some detail here since we are using properties of this functor to establish the stable version of this result. As a convention, our category of topological spaces $\topo$ is an abbreviation for the category of compactly generated weak Hausdorff spaces. We refer the reader to \cite[Appendix A]{SchwedeGlobal} for a thorough discussion with a slightly unusual terminology.

Our goal is now to define a Tamsamani $n$-groupoid corresponding to a given space. This requires some preparation. 

We define the cosimplicial space 
\[
\begin{tikzcd}[row sep=tiny]
	\Delta\ar[r, "\deltatop{-}"]		&\topo\\
	\left[m\right]\ar[mapsto]{r}		&\deltatop{m}
\end{tikzcd}
\]
where $\deltatop{0}=\{*\}$ and, for every positive integer $m$, we denote by $\deltatop{m}$ the standard topological $m$-simplex. The coface maps are inclusions of the corresponding face, and the codegeneracies are given by the corresponding maps collapsing the simplex to a lower-dimensional one. For more details, we refer the reader to \cite[Example 1.1]{GoerssJardine}. 

We can now define a functor $\R_n\colon\Delta^n\to \topo$, given for every $M=(m_1,\dots,m_n)\in\Delta^n$ by
$$\R_n(M)=\deltatop{m_n}\times\cdots\times \deltatop{m_1}.$$
Here, $\R$ stands for \lq\lq{}realization\rq\rq{} as this functor encodes a variant of geometric realization. A morphism in $\Delta^n$ is mapped by $\R_n$ to the corresponding product of induced morphisms between topological simplices. 

In this subsection, we recall the construction of Tamsamani's functors $\Pi_n\colon\topo\to \GTam$, discussed in \cite[\textsection 6]{Tamsamani}. An alternative approach is also given in \cite{BlancPaoli2014} and \cite{Paoli-2019}.

\begin{notation}
For any topological space $X$ and every $n\geq 1$, we write $\Hp^n(X)$ for the multi-simplicial set given by the composition $\Hom_{\topo}(\R_n(-),X)$, where $\Hom_{\topo}$ denotes the set of continuous maps. 
\end{notation}

\begin{rmk} \label{RmkHn}
Note that the functor $\Hp^n$ can be alternatively described as $\diag_* \Sing \colon \topo \to \funcat{n}{\set}$, where $\diag_*$ denotes the right adjoint to the diagonal functor. 
\end{rmk}

Let $v_i$ denote the $i$-th vertex of the topological simplex $\deltatop{m}$. Note that we use the reformulation of \cite[Prop.\ 6.3]{Tamsamani} rather than the original definition.

\begin{defn}
We define inductively the \emph{singular multicomplex} $\multising{n}{X}$ of a topological space $X$ to be a multi-simplicial subset $\multising{n}{X}\subset \Hp^n(X)$ such that
\begin{itemize}
\item $\multising{1}{X}=\Hp^1(X)$, and 
\item for $M\in\Delta^n$, the subset $\multising{n}{X}_{M}$ consists of the elements $f\in \Hp^n(X)_{M}$ for which the following holds: for every $0\leq k\leq n-1$, every $0\leq i\leq m_{n-k}$, and every 
\[
(x_n,\dots,\widehat{x_{n-k}},\dots,x_1)\in \deltatop{m_n}\times\cdots\times \widehat{\deltatop{m_{n-k}}}\times \cdots\times \deltatop{m_1},
\]
there is an element $f_i$ of $\multising{n-k-1}{X}_{m_1,\dots,m_{n-k-1}}$ independent of the variables \linebreak $x_n,\dots, x_{n-k+1}$ so that 
\[
f(x_n,\dots,x_{n-k+1},v_i, x_{n-k-1},\dots,x_1)=f_i(x_{n-k-1},\dots,x_1).
\] 
\end{itemize}
\end{defn}

The condition in the definition ensures both that the $0$-th level of Poincaré $n$-groupoid defined below is discrete and that the morphisms in the Poincaré $n$-groupoid can be composed.

We can define an equivalence relation on $\multising{n}{X}_M$ by saying that two elements $f$ and $g$ in $\multising{n}{X}_M$ are homotopic if and only if there exists $\gamma\in \multising{n+1}{X}_{M,1}$ such that $d_0(\gamma)=f$ and $d_1(\gamma)=g$ with $d_0, d_1\colon \multising{n+1}{X}_{M,1}\to \multising{n+1}{X}_{M,0}\cong\multising{n}{X}_M$. The aim is to identify two elements of $\multising{n}{X}_M$ when they are homotopic through maps with the same constraints as above.  

\begin{defn}
Given a topological space $X$ and a positive integer $n$, we define the \emph{Poincar\'e $n$-groupoid} of $X$ as the functor $\Pi_n(X)\colon(\Dop)^n\to\set$ obtained as the quotient of $\multising{n}{X}$ by the homotopy relation; in particular, we define $\Pi_n(X)_M$ to be the quotient of $\multising{n}{X}_M$ by the homotopy relation, for every $M\in \Delta^n$. 
\end{defn}

Tamsamani proved that $\Pi_n(X)$ is indeed what we call a Tamsamani $n$-groupoid \cite[Thm.\ 2.3.5]{TamsamaniPreprintSpaces}, \cite[Thm.\ 6.4]{Tamsamani}, and also that this defines a functor $\Pi_n(-)\colon\topo\to \GTam^n$.

\begin{ex}\label{ex_Poincare_gpd}
We will unravel the definition in the easiest cases. Fix a topological space $X$. We will write $\sim_h$ for the homotopy relation defined above. 

For $n=1$, the Poincar\'e $1$-groupoid of $X$ is in the first place a simplicial set with $m$-simplices given by $\Sing(X)_m/\sim_h$, where we use $\Hp^1(X)\cong \Sing(X)$ as in \cref{RmkHn}. We have to identify the equivalence relation $\sim_h$. 

Note that $\multising{2}{X}_{m,1}$ consists of continuous maps $H\colon \Delta[1] \times \Delta[m] \to X$ so that $H(-, v_i)$ is constant for all $v_i \in \Delta[m]$. In particular, if $m=0$, the relation $\sim_h$ on $\Sing(X)_0$ is just the identity and the objects of $\Pi_1(X)$ are precisely the points of $X$. For $m=1$, the relation $\sim_h$ is given by the endpoint-preserving homotopies between paths, i.e., between the elements of $\Sing(X)_1$. 

Since we already know that $\Pi_1(X)$ is the nerve of a groupoid, we conclude that it is the nerve of the fundamental groupoid, as the notation suggests. 
\end{ex}

\begin{defn} \label{DefGeomRealTamsamani}
The \emph{geometric realization functor} $|-|\colon\GTam^n\to \topo$ for Tamsamani $n$-groupoids is given by the following composition
\[
	\begin{tikzcd}
		\GTam^n\ar[r, hook]		&\left[(\Dop)^n,\set\right]\ar[r, "\diag"]		&\left[\Dop,\set\right]\ar[r, "|-|"]		&\topo
	\end{tikzcd}
\]
where $\diag\colon\left[(\Dop)^n,\set\right] \to \left[\Dop,\set\right]$ is given by $\diag(X)_r=X(r,\dots,r)$, for every $r\geq 0$, and $|-|\colon\left[\Dop,\set\right]\to\topo$ is the usual geometric realization.
\end{defn}

The homotopy hypothesis in Tamsamani's model for higher groupoids states the following (see \cite[Thm.\ 8.0]{Tamsamani}). 

\begin{thm}\label{UHH}
The functors $\Pi_n$ and $\vert -\vert$ restrict to
	\[\begin{tikzcd}
		{|-|\colon \GTam^n}\ar[r, shift left=1.1ex] & \topn\colon\Pi_n\ar[l, shift left=1.1ex] 
	\end{tikzcd}\]
and induce an equivalence of homotopy categories
\[\begin{tikzcd}
\Ho(\GTam^n)\ar[r, shift left=1.1ex] & \Ho(\topn)\ar[l, shift left=1.1ex] 
\end{tikzcd}\]
where $\Ho(\topn)$ is the localization of $\topn$ with respect to the usual topological weak equivalences and $\Ho(\GTam^n)$ is the localization of $\GTam^n$ with respect to the $n$-equivalences.
\end{thm}

\section{Modelling stable types}
\label{Sec:Spectra}

After discussing how we model higher categories in this paper, we want to explain how we model stable $n$-types so that both sides of the stable homotopy hypothesis acquire a fixed meaning. We will use sequential spectra to model the stable homotopy category. Moreover, to model connective spectra, it will turn out to be convenient to use Bousfield--Friedlander model of $\Gamma$-spaces \cite{BousfieldFriedlander}. We recall both notions and some of the comparison results; this part is purely expository. We then observe that the corresponding equivalences of homotopy categories restrict to an equivalence of models for stable $n$-types. 

\subsection{The model category of spectra}
Our spectra are going to take values in simplicial sets since this aligns nicely with the simplicial nature of our arguments. Before introducing spectra, we remind the reader of some standard facts about simplicial sets.

Recall that $\partial \Delta[n]$ denotes the boundary of an $n$-simplex. In what follows, we write $S^1=\Delta[1]/\partial \Delta[1]$. Moreover, we write $S^n=(S^1)^{\wedge n}$. 
 
\begin{defn} 
Given a pointed simplicial set $X\in\ssetp$, we define its simplicial \emph{loop space} by $\Omega X=\Hom_{\ssetp}(S^1\wedge\Delta[-]_{+},X)\in\ssetp$, and its \emph{suspension} by $\Sigma X= S^1\wedge X$. Here $(-)_{+}$ denotes the functor $\sset\to \ssetp$ that freely adjoins a basepoint.
\end{defn}

\begin{rmk}
One can verify that these functors form an adjunction $\Sigma\dashv\Omega$.
\end{rmk}

\begin{lem}\label{lem3}
The simplicial loop functor commutes with filtered colimits; that is, 
\[\Omega\colim_i X_i\cong\colim_i\Omega X_i.\] 
\end{lem}

\begin{proof}
Recall that every simplicial set with finitely many non-degenerate simplices (and thus, in particular, $S^1\wedge\Delta[-]_{+}$) is finite; see for example \cite[Lemma 3.1.2]{HoveyModelCategories}. 
\end{proof}

The reader should be aware that the simplicial loop space functor is not homotopically well-behaved in general. The situation changes if we restrict ourselves to Kan complexes, where we get the following well-known lemma. 

\begin{lem}\label{lem1}
If $X\in\ssetp$ is a Kan complex, then $\Omega\vert X\vert\simeq\vert\Omega X\vert$. 
\end{lem}

 \begin{proof}
We write $ \HOM_{\topo_*}$ for the internal hom in the category of pointed topological spaces. This lemma is a special case of the weak equivalence $\Sing \HOM_{\topo_*} (\lvert Y\rvert, \lvert X\rvert) \simeq \Hom_{\ssetp}(Y\wedge\Delta[-]_{+},X)$ for Kan complexes $X$ (see e.g.\ \cite[Prop.\ 1.1.11]{Hirschhorn}) together with the counit of the adjunction $\vert - \vert\dashv \Sing$ being a weak equivalence: 
\[
\Omega\lvert X\rvert\xleftarrow{\simeq}\left| \Sing \Omega\lvert X\rvert\right|=\left| \Sing \HOM_{\topo_*} (\lvert S^1\rvert, \lvert X\rvert)\right|  \simeq \left| \Hom_{\ssetp}(S^1\wedge\Delta[-]_{+},X)\right|=\left| \Omega X\right|
\]
\end{proof}

\begin{lem}\label{lem2}
For any space $Y\in\topo$, we have a natural isomorphism $\Omega\Sing Y\cong\Sing\Omega Y$. 
\end{lem}

\begin{proof} 
From the adjunctions 
\[\begin{tikzcd}
\ssetp\rar[shift left=0.9ex,"\Sigma_s"] & \ssetp\rar[shift left=1ex,"\vert - \vert"]\lar[shift left=0.8ex,"\Omega_s","\bot"'] & \topo_* \lar[shift left=0.8ex,"\Sing","\bot"']\rar[shift left=0.8ex,"\Sigma_t"] & \topo_* \lar[shift left=0.8ex,"\Omega_t","\bot"']
\end{tikzcd}\] 
and the well-known fact that $\Sigma$ and $\vert- \vert$ commute, we see that $$(\Omega_s\Sing)\vdash \left(\vert -\vert\Sigma_s\right)\cong\left(\Sigma_t\vert - \vert\right)\dashv\left(\Sing\Omega_t\right),$$ where the subscripts are meant to differentiate the simplicial and topological functors. Therefore, we conclude that $\Omega_s\Sing\cong\Sing\Omega_t$.
\end{proof}

We now return to the main focus of this subsection: the category of spectra.

\begin{defn}
A \emph{sequential spectrum} $X$ consists of a sequence $X^n\in\ssetp$ with $n\geq 0$ together with suspension maps $\sigma^n\colon \Sigma X^n\to X^{n+1}$ of pointed simplicial sets. A morphism $f\colon X\to Y$ of spectra consists of a sequence of maps $f^n\colon X^n\to Y^n$ in $\ssetp$ for each $n\geq0$ satisfying $\sigma^n(\id_{S^1}\wedge f^n)=f^{n+1}\sigma^n$.
\end{defn}

As we only work with sequential spectra, we will henceforth refer to them simply as spectra.
We denote the category of spectra by $\spectra$.

\begin{rmk}
Given spectra $X,Y$, it is possible to define a mapping space between spectra $\HOM_\spectra(X,Y)\in\ssetp$, given by 
$$\HOM_\spectra(X,Y)_n=\Hom_\spectra (X\wedge \Delta[n]_{+}, Y),$$
where the smash product of a spectrum with a simplicial set is defined levelwise. This provides the category of spectra with a structure of a simplicial category \cite[Prop.\ 2.2]{BousfieldFriedlander}. 
\end{rmk}

For a detailed treatment of models for spectra, we refer the reader e.g.\ to \cite{BousfieldFriedlander}, \cite{SchwedeSymmetricSpectra}, \cite{Switzer}. In this paper, we limit ourselves to the concepts and examples we will need for our proofs. 

\begin{ex}
The sphere spectrum, denoted by $\S$, is the spectrum with $\S^n=S^n$ and suspension maps $\Sigma S^n\xrightarrow{\cong} S^{n+1}$. 
\end{ex}

Recall that the homotopy groups of a pointed simplicial set can be defined as the homotopy groups of its geometric realization. This will be useful in the brief reminder below on weak equivalences between spectra. 

\begin{defn}
The \emph{$i$th homotopy group} of a spectrum $X$ is defined as $$\pi_i X=\colim_j \pi_{i+j} X^j,$$ where the colimit is taken over the maps 
\begin{center}
    \begin{tikzcd}
    \pi_{i+j} X^j\rar["\Sigma"] & \pi_{i+j+1}\Sigma X^j\rar["\sigma^j_*"] & \pi_{i+j+1}X^{j+1}
    \end{tikzcd}
\end{center} 
and runs over all $j\geq 0$ when $i\geq 0$, and over $j+i\geq 0$ for $i<0$.
\end{defn}

\begin{defn}
A spectrum $X$ is called \emph{connective} if $\pi_i X=0$ for $i<0$. A spectrum $X$ is an \emph{$\Omega$-spectrum} if the maps $\vert\sigma^n\vert^\flat\colon \vert X^n\vert\to \Omega \vert X^{n+1}\vert$ adjoint to the structure maps $\Sigma\vert X^n\vert\cong \vert\Sigma X^n\vert \xrightarrow{\vert\sigma^n\vert} \vert X^{n+1}\vert$ are weak equivalences in $\topo_*$. 

We denote the full subcategory of $\spectra$ whose objects consist of $\Omega$-spectra by $\Omega\spectra$. 
\end{defn}

\begin{rmk}\label{htpy_omega}
Note that the homotopy groups of $\Omega$-spectra are substantially easier to compute than those of ordinary spectra. Indeed, if $X$ is an $\Omega$-spectrum, then for each $i\geq 0$ and $j\geq 1$ we have 
$$\pi_{i+j} X^j =\pi_{i+j}\vert X^j\vert\cong \pi_{i+j-1} \Omega \vert X^j\vert\cong \pi_{i+j-1} \vert X^{j-1}\vert\cong\dots\cong \pi_i \vert X^0\vert =\pi_i X^0,$$ 
and this isomorphism is compatible with the maps in the colimit. Thus, $\pi_i X\cong \pi_i X^0$. Similarly, if $i<0$, we see that $\pi_i X\cong \pi_0 X^{-i}$. 
\end{rmk}

The category of spectra admits two natural model category structures, whose weak equivalences we now introduce.

\begin{defn}
A morphism of spectra $f\colon X\to Y$ is a \emph{strict weak equivalence} if each component $f^n\colon X^n\to Y^n$ is a weak equivalence in $\ssetp$. The morphism is a \emph{stable weak equivalence} if it induces isomorphisms $f_*\colon \pi_n X\xrightarrow{\cong} \pi_n Y$ for every $n$.
\end{defn}

According to \cite[Prop.\ 2.2]{BousfieldFriedlander}, there exists a model category structure on $\spectra$ having strict weak equivalences as its weak equivalences; this is called the \emph{strict model structure}, and we denote its homotopy category by $\Ho(\spectra)^\text{strict}$. Moreover, there exists another model structure on $\spectra$ whose weak equivalences are the stable weak equivalences, called the \emph{stable model structure}, and whose homotopy category we denote by $\Ho(\spectra)^\text{stable}$; this is shown in \cite[Prop.\ 2.3]{BousfieldFriedlander}. This is the stable homotopy category we are interested in. 

In fact, Bousfield--Friedlander construct the stable model category from the strict model category structure. In order to do so, they make use of a functor $Q\colon\spectra\to\Omega\spectra$ together with a natural transformation $\eta\colon 1\Rightarrow Q$ such that each component $\eta_X\colon X\to QX$ is a stable weak equivalence. Any such functor $Q$ satisfying certain technical conditions allows them to obtain a stable model category structure through a $Q$-localization theorem \cite[Thm.\ A.7]{BousfieldFriedlander}. 

For example, one could consider the functor $Q\colon \spectra\to\Omega\spectra$ given by 
$$(QX)^n=\colim_{i\geq 0} \Sing \Omega^i\vert X^{n+i}\vert,$$
where the sequential colimit is taken over the maps
$$\Sing\Omega^i\vert X^{n+i}\vert\to\Sing\Omega^{i+1}\vert X^{n+i+1}\vert$$
induced by the adjoints of the structure maps, 
$\vert X^{n+i}\vert\xrightarrow{\vert\sigma^{n+i}\vert^\flat} \Omega\vert X^{n+i+1}\vert$. 

Using \cref{lem2,lem3}, we can define the structure maps of $QX$ to be adjoint to the following isomorphisms of simplicial sets, where the first step is just re-indexing:
\[
\begin{tikzcd}[column sep=0.7cm]
\colim_{i\geq 0} \Sing \Omega^i\vert X^{n+i}\vert \arrow[r, "\cong"]& \colim_{i\geq 0} \Sing \Omega^{i+1}\vert X^{n+i+1}\vert \arrow[r, "\cong"] & \Omega\colim_{i\geq 0} \Sing \Omega^i\vert X^{n+1+i}\vert.
\end{tikzcd}
\] 
We observe that each level $(QX)^n$ is a filtered colimit of Kan complexes, thus a Kan complex itself. Hence, we can apply \cref{lem1} to conclude that $QX$ is an $\Omega$-spectrum.

For the natural transformation $\eta\colon \id\Rightarrow Q$, let $(\eta_X)^n\colon X^n\to (QX)^n$ be the colimit leg 
$$X^n\to\Sing\Omega^0 \vert X^{n+0}\vert\to\colim_{i\geq 0}\Sing\Omega^i \vert X^{n+i}\vert=(QX)^n,$$ 
where the first map is given by the unit $\id\Rightarrow\Sing\vert -\vert$. One can check that each $\eta_X$ is a morphism of spectra, and that $\eta$ is natural. Furthermore, we have an isomorphism 
\begin{align*}
    \pi_n X & = \colim_i \pi_{n+i} X^i\\
    & \cong \colim_i\pi_0\Omega^{n+i}\vert X^i\vert\\
    & \cong \colim_i\pi_0\Sing\Omega^{n+i}\vert X^i\vert\\
    & \cong \colim_i\pi_0\Omega^n\Sing\Omega^i\vert X^i\vert\\
    & \cong \pi_0\Omega^n\colim_i\Sing\Omega^i\vert X^i\vert\\
    & = \pi_0 \Omega^n(QX)^0\\
    & \cong \pi_n QX
\end{align*}
where we have used \cref{lem3}, \cref{htpy_omega}, the facts that $\pi_0$ is a left adjoint, and that $QX$ is a levelwise Kan $\Omega$-spectrum. As one can check that the map is precisely the one induced by $\eta_X$, we conclude that $\eta_X$ is a stable weak equivalence.

The following well-known lemma shows that stable weak equivalences and strict weak equivalences agree when working in $\Omega\spectra$.

\begin{lem}\label{stable_iff_strict} 
Let $f\colon X\to Y$ be a morphism in $\Omega\spectra$. Then $f$ is a stable weak equivalence if and only if it is a strict weak equivalence.
\end{lem} 

\begin{proof}
A strict weak equivalence always induces an equivalence on homotopy groups, so it is in particular a stable equivalence.

For the converse, assume $f$ is a stable weak equivalence; thus it induces isomorphisms $\pi_i X^0\cong\pi_i Y^0$ and $\pi_0 X^i\cong\pi_0 Y^i$ for all $i\geq 0$.

Now, consider $\pi_i X^n$ for any $i\geq 0, n\geq 0$; we claim that this can be reduced to one of the above cases. Indeed, if $n\geq i$, then 
\begin{align*}
    \pi_i X^n \cong \pi_0 \Omega^i \vert X^n\vert 
    \cong  \pi_0 X^{n-i},
\end{align*} 
and if $n\leq i$, then
\begin{align*}
    \pi_i X^n \cong  \pi_0 \Omega^i \vert X^n\vert 
    \cong  \pi_0 \Omega^{i-n}\Omega^n \vert X^n\vert
    \cong  \pi_0 \Omega^{i-n} \vert X^0\vert 
    \cong  \pi_{i-n} X^0
\end{align*} 
and thus one can see that $f$ induces isomorphisms on all homotopy groups.
\end{proof}

Let $\Ho(\Omega\spectra)^\text{strict}$ denote the full subcategory of $\Ho(\spectra)^\text{strict}$ whose objects are the $\Omega$-spectra. We can then show the following, as observed in \cite[\textsection 2.4]{BousfieldFriedlander}.

\begin{prop}\label{Ho_spectra_equiv_Ho_omega_spectra}
The functor $Q$ and the forgetful functor $U$ induce an equivalence 
\[\begin{tikzcd}
Q\colon \Ho(\spectra)^\text{stable}\rar[shift left=0.8ex] & \Ho(\Omega\spectra)^\text{strict}\colon U.\lar[shift left=0.8ex]
\end{tikzcd}\] 
\end{prop}

\begin{proof}
First, note that if $f\colon X\to Y$ is a stable weak equivalence in $\spectra$, then the diagram 
\[\begin{tikzcd}
    X\dar["f"']\rar["\eta_X"] & QX\dar["Qf"]\\
    Y\rar["\eta_Y"] & QY
\end{tikzcd}\] 
together with the fact that $\eta_X$ and $\eta_Y$ are stable weak equivalences, imply that $Qf$ is a stable weak equivalence. Hence, by  \cref{stable_iff_strict}, we see that $Q$ preserves weak equivalences and therefore induces a functor on homotopy categories. Similarly, the same lemma shows that $U$ preserves weak equivalences. 

To obtain the equivalence, we define $\eta\colon 1\Rightarrow UQ$ to be the unit (note that we slightly abuse the notation by sometimes making the forgetful functor explicit). The natural transformation $\eta$ is a degreewise stable weak equivalence and thus a natural isomorphism in $\Ho(\spectra)^\text{stable}$. For the counit, let $\epsilon\colon QU\Rightarrow 1$ be defined by $\epsilon_X=\eta_X^{-1}$; this inverse exists in $\Ho(\Omega\spectra)^\text{strict}$ since, by \cref{stable_iff_strict}, the maps $\eta_X$ are strict weak equivalences when restricted to $\Omega$-spectra $X$. This shows that we indeed obtain an equivalence of homotopy categories.
\end{proof}

\subsection{The model category of $\Gamma$-spaces}

Let $\Gamma$ denote the category of finite pointed sets and pointed maps. For $n\geq 0$, we denote by $\langle n\rangle$ the set $\{0,1,\dots,n\}$ with $0$ as the basepoint. We write $\und{n}$ for the subset of $\ord{n}$ consisting of $\{1,\ldots,n\}$. 

\begin{defn}
Let $\C$ be a pointed category, with $*$ as its zero object, i.e., an object which is both inital and terminal. A \emph{$\Gamma$-object} in $\C$ is a functor $\Gamma\to\C$ mapping $\langle 0\rangle$ to $*$. We denote the category of $\Gamma$-objects in $\C$ and natural transformations by $\Gamma\C$. In particular, a \emph{$\Gamma$-space} is a $\Gamma$-object in the category $\ssetp$. 
\end{defn}

We are interested in a specific type of $\Gamma$-spaces, the \emph{very special} ones, which we now recall. Since the definition makes sense in a broader context, we will extend it. We will later use it for Tamsamani $n$-categories as well.

\begin{defn} \label{DefSpecial}
Let $\C$ be a pointed category with a distinguished class $\clW$ of weak equivalences. 
A $\Gamma$-object $A$ in $\C$ is called \emph{special} if, for each $n\geq 0$, the map $$A\langle n\rangle \to A\langle 1 \rangle\times\dots\times A \langle 1 \rangle =A\langle 1 \rangle^{\times n},$$ induced by the indicator maps $\nu_j\colon\langle n\rangle\to\langle 1\rangle$ taking only the element $j$ to $1$ and all the other elements to $0$, is a weak equivalence in $\clW$.
\end{defn}

We will explain the relation between these $\nu_j$ and the maps $\nu^j$ from \cref{def-seg-map} later. For the moment, we acknowledge that the similarity in the notation is not a coincidence.

We call a $\Gamma$-space \emph{special} applying this definition to the category of pointed simplicial sets and weak equivalences. Note that applying this definition to the category of pointed sets with isomorphisms as equivalences gives us precisely an abelian monoid $A\langle 1\rangle$. This is the main motivation for the definition of special $\Gamma$-objects. The next remark exhibits a weaker but similar behaviour of $\Gamma$-spaces. 

\begin{rmk}\label{RmkMonoid}
For a special $\Gamma$-space $A$, the set $\pi_0 A\langle 1 \rangle$ becomes an abelian monoid, with multiplication 
$$\pi_0 A\langle 1 \rangle\times \pi_0 A\langle 1 \rangle\xleftarrow{(\nu_1, \nu_2)}\pi_0 A\langle 2 \rangle\xrightarrow{m}\pi_0 A\langle 1 \rangle$$ 
where $m\colon\langle 2\rangle\to\langle 1\rangle$ is the map in $\Gamma$ that sends $1,2$ to $1$, and the left-hand map is invertible by definition of a special $\Gamma$-space. 
\end{rmk}

\begin{defn}
A $\Gamma$-space $A$ is \emph{very special} if it is special, and the abelian monoid $\pi_0 A\langle 1\rangle$ is a group. 
\end{defn}

\begin{defn}
A morphism of $\Gamma$-spaces $f\colon A\to B$ is a \emph{strict weak equivalence} if each component $f_n\colon A\langle n\rangle\to B\langle n\rangle$ is a weak equivalence in $\ssetp$.
\end{defn}

\begin{thm}[{\cite[Thm.\ 3.5]{BousfieldFriedlander}}]
The category $\Gsset$ admits a model category structure, in which the weak equivalences are the strict weak equivalences. This is called the strict model structure, and we will denote its homotopy category by $\Ho\Gsset^\text{strict}$.
\end{thm}

\begin{rmk}
This model structure should not  be confused with the stable model structure on $\Gsset$, in which the weak equivalences are the maps that induce isomorphisms on the homotopy groups of the $\Gamma$-spaces, as opposed to these strict weak equivalences which induce isomorphisms on the homotopy groups of each of the component spaces.
\end{rmk}

Both fibrations and cofibrations admit an explicit description in this model structure, which we will not give here; all details can be found in \cite[\textsection 3]{BousfieldFriedlander}. 

The careful reader might note that strict weak equivalences in \cite{BousfieldFriedlander} involve a $\Sigma_n$-equivariance condition which we have chosen not to include in our definition above. The reason we can omit it is that this equivariance is not an additional property of a map of $\Gamma$-spaces, it is automatic. Indeed, the automorphism group of $\langle n\rangle$ in $\Gsset$ is precisely $\Sigma_n$, and so a natural transformation between functors from $\Gamma$ to $\ssetp$ must have $\Sigma_n$-equivariant maps as its component maps. The symmetric group needs to be involved, however, when defining the cofibrations.

\subsection{Equivalences of homotopy categories of spectra and $\Gamma$-spaces}

Any $\Gamma$-space $A\colon\Gamma\to\ssetp$ can be extended to a functor $A\colon \ssetp\to\ssetp$ (cf.\ \cite{BousfieldFriedlander}, \cite[Const.\ B.20]{SchwedeGlobal}) by letting, for each $K\in\ssetp$, 
$$A(K)=\int^{\langle n\rangle\in\Gamma} K^{\times n}\times A\langle n\rangle.$$ 

More explicitly, we can first extend the domain of $A$ to all pointed sets and obtain a functor $$A\colon\set_*\to\ssetp$$ by defining, for each $W\in\set_*$, $$A(W)=\colim_{V\subseteq W} A(V)$$ where $V$ ranges among finite pointed sets. It can then be further extended to a functor $$A\colon \ssetp\to\ssetp$$ by setting $(AK)_n=(AK_n)_n$ for $n\geq 0$ and $K\in\ssetp$, with obvious face and degeneracy maps.

Finally, our functor on pointed simplicial sets defines a functor on spectra $$A\colon \spectra\to\spectra$$ by defining, for a spectrum $X$, $(AX)^n=A(X^n)$, with suspension maps 
$$\Sigma A(X^n)\to A(\Sigma X^n)\to A(X^{n+1}),$$
where the first map is induced by the natural map of the form $K\wedge A(L)\to A(K\wedge L)$. (For a further explanation of this assembly map in the case of based topological spaces, we refer the reader to \cite[Prop.\ B.27]{SchwedeGlobal}.) 

The discussion in the above paragraphs allows us to define a functor $$S\colon \Gsset\to\spectra$$ mapping a $\Gamma$-space $A$ to the spectrum $A(\S)$, where $\S$ is the sphere spectrum. 
Moreover, there exists a functor in the other direction, $$\Phi(\S,-)\colon \spectra\to\Gsset$$ defined on each finite set $V$ and each spectrum $X$ by $\Phi(\S,X)(V)=\HOM(\S^V,X)$, where $\S^V=\S\times\S\times\dots\times\S$ is indexed by the non-basepoint elements in $V$.

\begin{lem}[{\cite[Lemma 4.6]{BousfieldFriedlander}}]\label{adjunction}
The functors $S$ and $\Phi(\S,-)$ are adjoints 
\[\begin{tikzcd}[column sep=small]
S\colon \Gsset\ar[rr, shift left=1.1ex] & \mbox{\scalebox{0.8}{$\bot$}} & \spectra\colon\Phi(\S,-).\ar[ll, shift left=1.1ex] 
\end{tikzcd}\]
\end{lem}

In fact, \cite[Lemma 4.6]{BousfieldFriedlander}  is more powerful, as it proves a simplicial version of this, which reduces to our statement by restricting to level 0. 

Once we have established these functors, it is possible to make the following definition.

\begin{defn}
The \emph{$i$th homotopy group} of a $\Gamma$-space $A$ is defined as $$\pi_i A=\pi_i SA;$$ that is, as the $i$th homotopy group of the spectrum obtained by applying the functor $S$ to $A$. 
\end{defn}

Recalling that the categories $\spectra$ and $\Gsset$ admit a strict model structure as described in the previous subsections,  it is not hard to show (see the discussion on p.\ 103 of \cite{BousfieldFriedlander}) that the adjunction $S\dashv\Phi(\S,-)$ from \cref{adjunction} is in fact a Quillen equivalence between the two strict model structures; in particular, we get equivalences of homotopy categories
\[\begin{tikzcd}
LS\colon \Ho(\Gsset)^\text{strict}\rar[shift left=0.8ex] & \Ho(\spectra)^\text{strict}\colon R\Phi(\S,-).\lar[shift left=0.8ex]
\end{tikzcd}\] 
Here $LS=S$ and $R\Phi(\S,X)=\Phi(\S,\hat{X})$ where $\hat{X}$ is a functorial fibrant replacement of $X$ in the strict model structure.

Furthermore, if we restrict these homotopy categories to the full subcategories given by very special $\Gamma$-spaces and connective $\Omega$-spectra, which we respectively denote as $\vsGsset$ and $\cospectra$, we obtain the following.

\begin{thm}[{\cite[Thm.\ 5.1]{BousfieldFriedlander}}]\label{BF}
The adjoint functors $LS\dashv R\Phi(\S,-)$ restrict to an adjoint equivalence 
\[\begin{tikzcd}
LS\colon \Ho(\vsGsset)^\text{strict}\rar[shift left=0.8ex] & \Ho(\cospectra)^\text{strict}\colon R\Phi(\S,-).\lar[shift left=0.8ex]
\end{tikzcd}\]
\end{thm}

In particular, we have that $SA$ is always an $\Omega$-spectrum when applied to a very special $\Gamma$-space $A$, and that $R\Phi(\S,X)=\Phi(\S,\hat{X})$ is a very special $\Gamma$-space when applied to a strictly fibrant $\Omega$-spectrum $X$. This allows us to deduce the following useful result.

\begin{prop}\label{htpy_level1}
If $A$ is a very special $\Gamma$-space, then $\pi_i A\cong\pi_i A\langle 1\rangle$.
\end{prop}

\begin{proof}
For each $i\geq0$, we see that 
    \begin{align*}
        \pi_i A  = \pi_i SA \cong \pi_i (SA)^0 = \pi_i A(\S)^0 = \pi_i A(S^0) =\pi_i A\langle 1\rangle,
    \end{align*}
where the second isomorphism uses \cref{htpy_omega} and the fact that $SA$ is an $\Omega$-spectrum by \cite[Pf.\ of Thm.\ 5.1]{BousfieldFriedlander}, and the equalities hold by definition. Note that $\pi_i A=\pi_i SA=0$ for $i<0$, since $SA$ is connective. 
\end{proof}

We wish to show that the above adjoint equivalence further restricts to $n$-types on both sides. Let $\ospectran$ denote the full subcategory of connective $\Omega$-spectra whose homotopy groups are concentrated in $[0,n]$. Similarly, let $\vsGssetn$ denote the full subcategory of very special $\Gamma$-$n$-types; that is, the subcategory of very special $\Gamma$-spaces $A$ such that the homotopy groups of each space $A\langle m\rangle$ are concentrated in $[0,n]$. This is indeed unambiguous, as we observe next.

\begin{rmk}\label{htpy_special}
If $A$ is a special $\Gamma$-space, then for any $m\geq 0$, the homotopy groups of the pointed space $A\langle m\rangle$ are concentrated in $[0,n]$ if and only if the homotopy groups of $A\langle 1 \rangle$ are.

Together with \cref{htpy_level1}, this implies that our category $\vsGssetn$ is the same as the category $(\vsGsset)_{[0,n]}$.
\end{rmk}

With this phrasing, we can now formulate the following result.
 
\begin{thm}\label{equiv_gamma_and_spectra}
The adjoint functors $LS\dashv R\Phi(\S,-)$ restrict to an adjoint equivalence 
\[\begin{tikzcd}
LS\colon \Ho(\vsGssetn)^\text{strict}\rar[shift left=0.8ex] & \Ho(\ospectran)^\text{strict}\colon R\Phi(\S,-).\lar[shift left=0.8ex]
\end{tikzcd}\]
\end{thm}

\begin{proof}
Let $A\in \vsGssetn$, and consider $SA\in \cospectra$ (by \cite[Lemma 4.2]{BousfieldFriedlander}). We have 
    \begin{align*}
        \pi_i SA  =\pi_i A  \cong \pi_i A\langle 1\rangle,
    \end{align*}
where the equality is the definition of the homotopy groups of a $\Gamma$-space, and the isomorphism is due to \cref{htpy_level1}. Since $A$ was assumed to take values in ${\ssetp}_{[0,n]}$, we know that the homotopy groups of $A\langle 1\rangle$ are concentrated in $[0,n]$; thus, the same is true for the homotopy groups of the spectrum $SA$.

For the other direction, let $X\in \ospectran$, and consider $R\Phi(\S,X)=\Phi(\S,\hat{X})\in \vsGsset$ for a fibrant replacement $\hat{X}$ of $X$. Because of \cref{htpy_special}, we know that for each $m$, the pointed space $\Phi(\S,\hat{X})\langle m\rangle$ will have homotopy groups concentrated in $[0,n]$ precisely when $\Phi(\S,\hat{X})\langle 1\rangle$ does. In addition,
    \begin{align*}
        \pi_i\Phi(\S,\hat{X})\langle 1\rangle \cong \pi_i \Phi(\S,\hat{X}) =\pi_i S\Phi(\S,\hat{X}) \cong \pi_i \hat{X}\cong \pi_i X,
    \end{align*}
where the first isomorphism is due to \cref{htpy_level1}, and the equality is by definition of the homotopy groups of the $\Gamma$-space $\Phi(\S,\hat{X})$. The last two isomorphisms are induced by strict weak equivalences $$S\Phi(\S,\hat{X})\xrightarrow{\sim}\hat{X} \xleftarrow{\sim} X$$ given by the counit of the derived adjunction and by the fibrant replacement (recall that strict weak equivalences induce isomorphisms on homotopy groups).
\end{proof}

\begin{cor}\label{half_of_SHH_from_BF}
There exists an equivalence of homotopy categories
\[\begin{tikzcd}
U\circ LS\colon \Ho(\vsGssetn)^\text{strict}\rar[shift left=0.8ex] & \Ho(\spectran)^\text{stable}\colon R\Phi(\S,-)\circ Q.\lar[shift left=0.8ex]
\end{tikzcd}\]
\end{cor}

\begin{proof}
This is an immediate consequence of \cref{equiv_gamma_and_spectra} and \cref{Ho_spectra_equiv_Ho_omega_spectra}, once we recall that $\eta\colon 1\Rightarrow Q$ is componentwise a stable weak equivalence and thus $Q\colon\spectra\to\Omega\spectra$ restricts to $Q\colon \spectran\to\ospectran$.
\end{proof}

\section{Stable homotopy hypothesis}
\label{Sec:SHH}

The aim of this section is to put together all the ingredients prepared so far. We first need to define the objects which will play the role of fully groupoidal weak symmetric monoidal $n$-categories. The rest of the section will be devoted to proving the stable homotopy hypothesis.

\subsection{Picard--Tamsamani $n$-categories}

We now define the main player in our version of the stable homotopy hypothesis. As mentioned in the introduction, the stable homotopy hypothesis should relate stable homotopy $n$-types with fully groupoidal symmetric monoidal weak $n$-categories. Since both the definition of a symmetric monoidal bicategory and of a tricategory including all coherences is quite challenging, an explicit fully algebraic definition of symmetric monoidal weak $n$-categories for all $n$ seems out of reach at the moment. However, once we decide to use Segal-type models, the definition ends up being fairly simple.

Before we can phrase the definition, recall that there is a functor relating the categories $\Dop$ and $\Gamma$. More precisely, there is a functor $\phi\colon \Dop\to \Gamma$ given by $\phi([n])=\ord{n} $, for all $n\geq 0$, and which sends an order-preserving map $\alpha\colon [m]\to [n]$ to the pointed map 
\[ \phi(\alpha)\colon \langle n\rangle \to \langle m\rangle, \;\;
j\mapsto \begin{cases} 
i & \text{if } j\in [\alpha(i-1)+1, \alpha(i)], \\
0 & \text{else}.
\end{cases} \] 

\begin{ex}
Recall the maps $\nu^j\colon [1]\to [n]$ from \cref{def-seg-map} of the Segal maps. We find that $\phi(\nu^j)=\nu_j\colon \ord{n}\to \ord{1}$, justifying the notation. In particular, the underlying simplicial object of a $\Gamma$-object has its Segal maps of \cref{def-seg-map} being weak equivalences if and only if it is special in the sense of \cref{DefSpecial}.

Moreover, we have $\phi(d^1)=m$ for $d^1\colon [1] \to [2]$ and the map $m$ from \cref{RmkMonoid}. Thus, the monoid structure on $\pi_0A\ord{1}$ is determined by the underlying simplicial object of a $\Gamma$-space $A$. 
\end{ex}

The functor $\phi$ makes it possible to define the following:

\begin{defn}\label{DefUnderlyingDelta}
Let $A\colon \Gamma \to \C$ be a $\Gamma$-object in a pointed category $\C$. Then we call the composite functor $\Dop \xrightarrow{\phi} \Gamma \xrightarrow{A} \C$ the \emph{underlying simplicial object} of $A$.
\end{defn}

Now we are ready to define our fully groupoidal, symmetric monoidal weak $n$-categories.

\begin{defn}\label{Definition PicTam}
A \emph{Picard--Tamsamani $n$-category} is a functor $\XX\colon \Gamma \times \Dsop{n} \to \set$ such that the restriction to $\Dsop{n+1}$ is a Tamsamani $(n+1)$-groupoid with one object.
\end{defn}

We denote the category of Picard--Tamsamani $n$-categories and all natural transformations between them 
by $\PicGTam^n$. 

\begin{rmk}
The reader should be aware that some authors (e.g.\ \cite{JohnsonOsorno}) use the notion of a Picard \emph{groupoid} for what we would call Picard category, to emphasize the invertibility of the morphisms. 
\end{rmk}

Next, we define the appropriate notion of equivalence between Picard--Tamsamani $n$-categories.

\begin{defn}\label{equivPic} A morphism of Picard--Tamsamani $n$-categories $f\colon \XX\to \YY$ is a \emph{weak equivalence} if the induced map between the restrictions of $\XX$ and $\YY$ to $\Dnop$ is an $(n+1)$-equivalence of Tamsamani $(n+1)$-groupoids.
\end{defn}

Recall that there is a notion of $\pi_0$ for Tamsamani $n$-groupoids. Its properties, discussed in \cref{PropertiesTamsamani}, allow us to mimic the discussion for spaces and define very special $\Gamma$-objects in Tamsamani $n$-groupoids as follows.

\begin{rmk}
For a special $\Gamma$-object $A$ in Tamsamani $n$-groupoids, the set $\pi_0 A\langle 1 \rangle$ becomes an abelian monoid, with multiplication 
$$\pi_0 A\langle 1 \rangle\times \pi_0 A\langle 1 \rangle\xleftarrow{(\nu_1, \nu_2)}\pi_0 A\langle 2 \rangle\xrightarrow{m}\pi_0 A\langle 1 \rangle$$ 
where $m\colon\langle 2\rangle\to\langle 1\rangle$ is the map that sends $1,2$ to $1$, and the left-hand map is invertible by definition of a special $\Gamma$-object. 
\end{rmk}

\begin{defn}
A $\Gamma$-object in Tamsamani $n$-groupoids $A$ is \emph{very special} if it is special, and the abelian monoid $\pi_0 A\langle 1\rangle$ is a group.
\end{defn}

\begin{lem}\label{PicTam_are_vsGammaGTam}
A functor $\XX\colon \Gamma \times (\Delta^{\op})^n \to \set$ is a Picard--Tamsamani $n$-category if and only if it lifts to a very special $\Gamma$-object in Tamsamani $n$-groupoids. 
\end{lem}

\begin{proof}
Assume we are given a functor $\XX\colon \Gamma \times \Dsop{n} \to \set$. We can also consider $\XX$ as a functor $\Gamma \to [\Dsop{n}, \set]$. Recall that we write $\XX_n$ for $\XX\langle n\rangle$ when we consider $\XX$ as its underlying $\Dsop{}$-object.

Assume first that $\XX$ is a Picard--Tamsamani $n$-category. As such, it takes by assumption values in Tamsamani $n$-groupoids by \cref{Definition GTa,,Definition PicTam}. Moreover, note that the condition of being special only concerns the underlying $\Delta^{\op}$-object and is satisfied by the definition since $\XX$ is a Tamsamani $(n+1)$-groupoid with one object.

Next, we observe that the monoid structure of $\pi_0\XX_1$ again only depends on the underlying $\Delta^{\op}$-object, and thus so does the question of existence of inverses. So we are left to show that $\pi_0\XX_1$ is a group for a Tamsamani $(n+1)$-groupoid $\XX$. Recall from \cref{lem-multsimpsets-2R} that $(p^{(1)}\XX)_1=p^{(0)}\XX_1=\pi_0\XX_1$. Moreover, observe that we can conclude inductively from \cref{Definition GTa} that $p^{(1)}\XX$ is the nerve of a groupoid. Applying \cref{lem-multsimpsets-2R}, we see that $(p^{(1)}\XX)_0=p^{(0)}\XX_0=*$ since a Picard--Tamsamani $n$-category has only one object by definition. Thus, the groupoid $p^{(1)}\XX$ is actually the group $\pi_0\XX_1$, so that $\XX$ is indeed a very special $\Gamma$-object in Tamsamani $n$-groupoids.

Conversely, assume now $\XX$ to be a very special $\Gamma$-object in Tamsamani $n$-groupoids. In particular, it has $\XX_0=*$ and is discrete. Moreover, since the Segal condition comes precisely from being special, we can conclude by \cref{DefTamsamani} that the underlying $\Dsop{n+1}$-object of $\XX$ is indeed a Tamsamani $(n+1)$-category with one object. We still have to check that it is a Tamsamani $(n+1)$-groupoid. 

Since $\XX_k$ is a Tamsamani $n$-groupoid for all $k\geq 0$ by the assumption of the lemma, we only need to check that $p^{(n)}\XX$ is a Tamsamani $n$-groupoid by \cref{Definition GTa}. By \cref{lem-multsimpsets-2R}, we have $(p^{(n)}\XX)_k=p^{(n-1)}\XX_k$ for all $k\geq 0$, and this is a Tamsamani $(n-1)$-groupoid since $\XX_k$ is a Tamsamani $n$-groupoid. Repeating this argument, we are left to show that $p^{(1)}\XX$ is a groupoid. 

Now we can reverse the argument above. We already know that $p^{(1)}\XX$ is a category with one object, and moreover its set of morphisms is precisely $(p^{(1)}\XX)_1=p^{(0)}\XX_1=\pi_0\XX_1$. Since this monoid was assumed to be a group, we have indeed shown that $p^{(1)}\XX$ is a groupoid. This completes the proof. 
\end{proof}

\begin{rmk}\label{we_PicTam_vsGammaGTam}
By \cref{equiv_Tam_cats_woo} and \cref{PicTam_are_vsGammaGTam}, a morphism $f\colon \XX\to \YY$ is a weak equivalence of Picard--Tamsamani $n$-categories if and only if each component $f_m\colon \XX\langle m\rangle\to \YY\langle m\rangle$ is an $n$-equivalence in $\GTam^n$.
\end{rmk}

\subsection{Promoting Tamsamani's homotopy hypothesis to $\Gamma$-objects}
Since we opted to work with homotopy types seen in the category of simplicial sets instead of topological spaces, it will be necessary to consider the slightly modified functors shown below. 
\[
	\begin{tikzcd}
		\GTam^n\ar[r, hook]\ar[rr, dashed, bend left=20, "B"]\ar[rrr, bend left=30, "|-|"]		&\left[(\Dop)^n,\set\right]\ar[r, "\diag"]		&\left[\Dop,\set\right]\ar[r, "|-|"]\ar[ll, dashed, bend left=20, "\Pi_n\circ|-|=P_n"]		&\topo\ar[lll, bend left=35, "\Pi_n"]
	\end{tikzcd}
\]

In what follows, we show that the functors 
\[\begin{tikzcd}[column sep=large]
\GammaGTam	\ar[r, shift left=1.1ex, "B_*"] &	\GammasSet\ar[l, shift left=1.1ex, "(P_n)_*"] 
\end{tikzcd}\] 
restrict to very special $\Gamma$-objects, and that they induce an equivalence of the homotopy categories if the weak equivalences are defined levelwise.

We begin by studying the compatibility of the $0$-th homotopy groups with the two functors defined above.

\begin{lem}\label{pi0_GTam_compatibility}Let $X$ be a topological space, and let $\XX$ be a Tamsamani $n$-groupoid. Then the following holds:
    \begin{enumerate}
        \item\label{part_1} There are natural isomorphisms $\pi_0(X)\cong \pi_0(\Pi_1(X))$ and $p^{(i)}(\Pi_n(X))\cong\Pi_i(X)$ for all $1\leq i\leq n-1$. In consequence, $\pi_0(X)\cong\pi_0(\Pi_n(X))$.
        \item\label{part_2} There is a natural isomorphism $\pi_0(\XX)\cong\pi_0(|\XX|)$.
    \end{enumerate}
\end{lem}

\begin{proof} 
The isomorphism $\pi_0(X)\cong \pi_0(\Pi_1(X))$ follows from \cref{ex_Poincare_gpd}, and the compatibility between the truncation functor on Tamsamani groupoids and Poincar\'e $n$-groupoids is proven in \cite[Thm.\ 2.3.5]{TamsamaniPreprintSpaces}, where the truncation functor $p^{(n-1)}$ is denoted by $T$. 

Therefore, we have the following chain of isomorphisms
\begin{align*}
\pi_0(X)\cong	&p(\Pi_1(X))\\
        \cong   &p^{(0)}p^{(1)}(\Pi_n(X))\\
        \cong   &p(\Pi_n(X))\\
        =       &\pi_0(\Pi_n(X))
\end{align*}
which concludes the proof of (\cref{part_1}).

For the proof of (\ref{part_2}), we recall that, in \cite[\textsection 11]{Tamsamani}, Tamsamani defines a natural $n$-equivalence $\XX\to\Pi_n(|\XX|)$, which by \cref{n_equiv_pi0} induces an isomorphism of $0$-th homotopy groups. To complete the proof, we compose this 
isomorphism with the one obtained in (\ref{part_1}) between $\pi_0(\Pi_n(|\XX|))$ and $\pi_0(|\XX|)$.
\end{proof}

\begin{rmk} 
\cref{pi0_GTam_compatibility} is an instance of a more general principle: The functors $p^{(r)}$ can be seen as $r$-th Postnikov truncations, as explained in more detail in \cite{BlancPaoli2014}.
\end{rmk}

\begin{lem}\label{Gamma_GTam_to_Gamma_sset} If $\XX\colon\Gamma\to\GTam^n$ is a very special $\Gamma$-object, then so is $B\XX\colon\Gamma\to \sset$. 
\end{lem}

\begin{proof}
We know that $B\XX\langle 0\rangle=\ast$ because $\XX\langle 0\rangle$ is discrete at $\ast\in \set$. Since $\XX$ is very special, the map $\XX\langle n\rangle\to \XX\langle 1\rangle^n$ is an $n$-equivalence. The geometric realization $B$ sends $n$-equivalences to weak homotopy equivalences; this follows from the fact that Tamsamani's geometric realization does so (see \cite[Prop.\ 11.2]{Tamsamani}) and that the classical geometric realization $|-|\colon\sset\to\topo$ preserves and reflects weak equivalences. The functor $B$ also preserves products as $\diag$ is a right adjoint and the classical geometric realization preserves products. Hence $B\XX$ satisfies the Segal condition too. 

Finally, by \cref{pi0_GTam_compatibility}, we know that there is a natural isomorphism $\pi_0(B\XX\langle 1\rangle)\cong\pi_0(\XX\langle 1\rangle)$, which is compatible with the monoid structure coming from the Segal condition and thus allows us to deduce that $\pi_0(B\XX\langle 1\rangle)$ is an abelian group since $\pi_0(\XX\langle 1\rangle)$ is one.
\end{proof}

\begin{lem}\label{Pi_preserves_prod} The functor $\Pi_n\colon\topo\to \GTam^n$ preserves finite products up to $n$-equivalence.
\end{lem}

\begin{proof}
We begin by observing that there exists a zigzag 
of homotopy weak equivalences $|\Pi_n(Z)|\to Z$ for every $Z\in \topo$, functorial in $Z$, see \cite[Rmk.\ 10.8]{Tamsamani}. In consequence, given $X,Y\in \topo$, we have 
     \[\begin{tikzcd}[column sep=small]
     \left|\Pi_n(X\times Y)\right|\ar[rr]\ar[dr, "\sim"']		&		&\left|\Pi_n(X)\right|\times \left|\Pi_n(Y)\right|\ar[dl, "\sim"]\\
 \ 		&X\times Y
     \end{tikzcd}\]
This means that the map $\Pi_n(X\times Y)\to\Pi_n(X)\times\Pi_n(Y)$ induces a weak homotopy equivalence after realization. 

Applying $\Pi_n$ again, we obtain an $n$-equivalence by \cite[Prop.\ 11.4]{Tamsamani}. Together with the natural $n$-equivalence $\XX\to\Pi_n(|\XX|)$ of \cite[\textsection 11]{Tamsamani}, we get the commutative square
\[
\begin{tikzcd}
\Pi_n(X\times Y)\arrow[d, "\sim"]\arrow[r] & \Pi_n(X)\times\Pi_n(Y)\arrow[d, "\sim"] \\
\Pi_n\left|\Pi_n(X\times Y)\right|\arrow[r, "\sim"] & \Pi_n\left|\Pi_n(X)\times\Pi_n(Y)\right|
\end{tikzcd}
\]
showing that the upper horizontal map is an $n$-equivalence, which concludes the proof.
\end{proof}

\begin{cor}\label{Gamma_sset_to_Gamma_GTam} If $A\colon\Gamma\to \sset$ is a very special $\Gamma$-object, then so is $P_n A\colon\Gamma\to \GTam^n$.
\end{cor}

\begin{proof}
By \cref{Pi_preserves_prod}, it only remains to prove that, when $A\colon\Gamma\to \sset$ is very special, the abelian monoid $\pi_0(P_n A\langle 1\rangle)$ is a group. Observe that we have
    \begin{align*}
        \pi_0(A\langle 1\rangle)\cong 	&\pi_0(|A\langle 1\rangle|)\\
        \cong & \pi_0(\Pi_n(|A\langle 1\rangle|))\\
        \cong & \pi_0(P_n A\langle 1\rangle)
    \end{align*} 
where the second isomorphism is given by \cref{pi0_GTam_compatibility}, which is compatible with the monoid structure given by the Segal condition. Since $\pi_0(A\langle 1\rangle )$ is an abelian group, so is $\pi_0(P_n A\langle 1\rangle)$.
\end{proof}

\begin{rmk}\label{sset_ssetp}When we consider $\Gamma$-objects in $\sset$, the condition that $\langle 0\rangle$ in $\Gamma$ needs to be sent to the $*$ makes the remaining levels into pointed spaces. Moreover, given a $\Gamma$-object on $\ssetp$, the condition that $\langle 0\rangle$ in $\Gamma$ needs to be sent to the zero object of $\ssetp$ determines the pointedness in the remaining levels. Therefore, the categories $\Gamma\ssetp$ and $\Gamma\sset$ are equivalent. Note that the obvious equivalence preserves very special $\Gamma$-objects.
\end{rmk}

\begin{thm}\label{half_of_SHH_from_UHH} The functors $P_n$ and $B$ induce an equivalence of homotopy categories
        \[\begin{tikzcd}
    \Ho(\mathrm{v.s.} \Gamma\GTam^n)\ar[r, shift left=1.1ex] & \Ho(\vsGssetn)^\text{strict}\ar[l, shift left=1.1ex] 
    \end{tikzcd}\]
where the category $\Ho(\mathrm{v.s.} \Gamma\GTam^n)$ is the localization with respect to levelwise $n$-equi\-valences. 
\end{thm}

\begin{proof}
Since the geometric realization $|-|\colon\sset\to\topo$ preserves and reflects weak equivalences, \cref{Gamma_GTam_to_Gamma_sset} and  \cref{Gamma_sset_to_Gamma_GTam} allow us to conclude that 
the functors 
\[\begin{tikzcd}[column sep=large]
\GammaGTam	\ar[r, shift left=1.1ex, "B_*"] &	\GammasSet\ar[l, shift left=1.1ex, "(P_n)_*"] 
\end{tikzcd}\] 
restrict to very special $\Gamma$-objects.   
Moreover, the induced functors on the homotopy categories are inverse equivalences, using
\cref{UHH} and \cref{sset_ssetp}.
\end{proof}

\subsection{Proof of the stable homotopy hypothesis}
The purpose of this subsection is to finally combine all the ingredients and to prove the stable homotopy hypothesis in the Tamsamani model. 

In the diagram below we present a sketch of the proof, which consists of using modifications of Tamsamani's unstable homotopy hypothesis (UHH) and of the result by Bousfield and Friedlander (BF) to relate Picard--Tamsamani categories with stable types through very special $\Gamma$-spaces. We denote by $\simeq_{\text{Ho}}$ the equivalence of homotopy categories.
\begin{tz} 
\def\r{5cm}
\node (B1) at (0,2) {$\vsGsset$};
\node (C1) at (\r,2) {$\Omega\spectra_{\geq 0}$};
\node (A2) at (-\r,0) {$\vsGammaGTam^n$};
\node (B2) at (0,0) {$\vsGssetn$};
\node (C2) at (\r,0) {$\spectran$};
\node (A3) at (-\r,-2) {$\GTam^n$};
\node (B3) at (0,-2) {$\ssetpn$};

\draw[->] (B1) to node[above,scale=0.8] {$\simeq_{\text{Ho}}$} (C1);
\draw[] (B1) to node[below,scale=0.8] {BF (Thm.\ \ref{BF})} (C1);
\draw[->] (A2) to node[above,scale=0.8] {$\simeq_{\text{Ho}}$} (B2);
\draw[] (A2) to node[below,scale=0.8] {Thm.\ \ref{half_of_SHH_from_UHH}} (B2);
\draw[->] (B2) to node[above,scale=0.8] {$\simeq_{\text{Ho}}$} (C2);
\draw[] (B2) to node[below,scale=0.8] {Cor.\ \ref{half_of_SHH_from_BF}} (C2);
\draw[->] (A3) to node[above,scale=0.8] {$\simeq_{\text{Ho}}$} (B3);
\draw[] (A3) to node[below,scale=0.8] {UHH (Thm.\ \ref{UHH})} (B3);

\node[scale=1.4, rotate=90] at (-2.5,-1) {$\leadsto$};
\node[scale=1.4, rotate=-90] at (2.7,1) {$\leadsto$};
\end{tz}
We consider here the stable homotopy categories for spectra. 

\begin{thm}[Stable homotopy hypothesis]\label{StableHH}
There is an equivalence between the homotopy categories of $\PicGTam^n$ with $n$-equivalences and of $\spectran$ with stable equivalences. 
\end{thm}

\begin{proof}
We have presented a dissected proof in the previous sections, and will now assemble the pieces.
    
By \cref{PicTam_are_vsGammaGTam} and \cref{we_PicTam_vsGammaGTam}, we know that there is an equivalence 
    \[
    \Ho(\PicGTam^n)\simeq\Ho(\vsGammaGTam^n).
    \]
Now, from Tamsamani's unstable homotopy hypothesis (see \cref{UHH}), we have shown in \cref{half_of_SHH_from_UHH} that there is an equivalence
    \[
    \Ho(\vsGammaGTam^n)\simeq\Ho(\vsGssetn),
    \] 
and, from Bousfield-Friedlander's \cref{BF}, we concluded in \cref{half_of_SHH_from_BF} that there is an equivalence
    \[
    \Ho(\vsGssetn)\simeq\Ho(\spectran).
    \]
This completes the proof.
\end{proof}

\section{Segal's $K$-theory functor vs Lack--Paoli's $2$-nerve functor}

In this section, we want to compare Segal's $K$-theory functor to Lack--Paoli's $2$-nerve. To make them comparable, we restrict ourselves to Picard categories and view them either as a special case of a symmetric monoidal category or as a bicategory with one object. We will show that the underlying simplicial objects in groupoids of the two constructions are naturally  equivalent in this case. This will be helpful both in spirit and in the technical sense for the comparison of our stable homotopy hypothesis to the classical result for Picard categories.

\subsection{Segal's $K$-theory functor for Picard categories} \label{KP}
Recall that, for $n\geq 0$, we write $\ord{n}$ for the set $\{0,1,\ldots,n\}$ in $\Gamma$ with basepoint $0$, and $\und{n}$ for the subset of $\ord{n}$ consisting of $\{1,\ldots,n\}$. 

We consider a $K$-theory construction for symmetric monoidal categories due to Segal \cite{SegalCat}. One of the original motivations for this construction was to show that the algebraic $K$-theory of a ring is an infinite loop space. In fact, it was for this reason that Segal first defined $\Gamma$-spaces. Mandell describes Segal's $K$-theory functor $\mathcal{K}\colon \Sym\to \GammaCat$ from the category of symmetric monoidal categories and strictly unital op-lax morphisms to the category of functors from $\Gamma$ to $\cat$ in \cite[Constr.\ 3.1, Variant 3.4]{MandellKtheory}.

We consider the full subcategory $\Pic$ of $\Sym$ consisting of the \emph{Picard categories}, which are symmetric monoidal groupoids with every object invertible with respect to the monoidal structure. We explicitly describe the image of a Picard category under the functor $\mathcal K$, which lands in $\GammaGpd$, the category of functors from $\Gamma$ to $\gpd$. 

Let $(P,\otimes,\mathbf{1})$ be a Picard category. Segal's $K$-theory construction $\KP\colon\Gamma\to \gpd$ for $P$ is defined as follows. 

We first describe $\KP$ on objects. The category $\KP\ord{0}=*$ is the terminal category. For $n\geq 1$, an object in the category $\KP\ord{n}$ is a collection $\{x_I, f_{I,J}\}$ consisting of
\begin{abc}
\item for all $I\subset \und{n}$, an object $x_I\in P$,
\item for all disjoint $I,J\subset \und{n}$, an isomorphism $f_{I,J}\colon x_{I\sqcup J}\to x_I\otimes x_J$ in $P$, 
\end{abc}
such that the following conditions hold:
\begin{rome}
\item $x_{\emptyset}=\mathbf{1}$ and $f_{\emptyset,J}\colon x_J\to \mathbf{1}\otimes x_J$ is the unit isomorphism, for all $J\subset \und{n}$,
\item for all disjoint $I,J\subset \und{n}$, the following diagram commutes, 
\begin{tz}
\node (A) at (0,0) {$x_{I\sqcup J}$};
\node (B) at (0,-2) {$x_{J\sqcup I}$};
\node (C) at (2.5,0) {$x_I\otimes x_J$};
\node (D) at (2.5,-2) {$x_J\otimes x_I$};

\draw[d] (A) to (B);
\draw[->] (A) to node[above,scale=0.8] {$f_{I,J}$} (C);
\draw[->] (B) to node[below,scale=0.8] {$f_{J,I}$} (D);
\draw[->] (C) to node[right,scale=0.8] {$\gamma$} (D);
\end{tz}
where $\gamma$ denotes the symmetry isomorphism of $P$,
\item for all mutually disjoint $I,J,K\subset \und{n}$, the following diagram commutes,
\begin{tz}
\node (A) at (0,0) {$x_{I\sqcup J\sqcup K}$};
\node (B) at (0,-2.5) {$x_I\otimes x_{J\sqcup K}$};
\node (C) at (3.5,0) {$x_{I\sqcup J}\otimes x_K$};
\node (D) at (3.5,-2.5) {$x_I\otimes (x_J\otimes x_K)$};
\node (E) at (6,-1.25) {$(x_I\otimes x_J)\otimes x_K$};

\draw[->] (A) to node[left,scale=0.8] {$f_{I,J\sqcup K}$} (B);
\draw[->] (A) to node[above,scale=0.8] {$f_{I\sqcup J,K}$} (C);
\draw[->] (B) to node[below,scale=0.8] {$\id\otimes f_{J,K}$} (D);
\draw[->] (D) to node[below,scale=0.8] {$\;\;\alpha$} (E);
\draw[->] (C) to node[above,scale=0.8] {$\;\;\;\;\;\;\;\;\;f_{I,J}\otimes \id$} (E);
\end{tz}
where $\alpha$ denotes the associativity isomorphism of $P$. We denote this composite by $f_{I,J,K}$.
\end{rome}

\noindent A morphism $H\colon \{x_I,f_{I,J}\}\to \{y_I,g_{I,J}\}$ is given by a collection of morphisms $\{H_I\colon x_I\to y_I\}$ in $P$, indexed by $I\subset \und{n}$, such that the following conditions hold: 
\begin{rome}
\item $H_\emptyset$ is the identity at $\mathbf{1}$, 
\item for all disjoint $I,J\subset \und{n}$, the following diagram commutes.
\begin{tz}
\node (A) at (0,0) {$x_{I\sqcup J}$};
\node (B) at (0,-2) {$y_{I\sqcup J}$};
\node (C) at (2.5,0) {$x_I\otimes x_J$};
\node (D) at (2.5,-2) {$y_I\otimes y_J$};

\draw[->] (A) to node[left,scale=0.8] {$H_{I\sqcup J}$} (B);
\draw[->] (A) to node[above,scale=0.8] {$f_{I,J}$} (C);
\draw[->] (B) to node[below,scale=0.8] {$g_{I,J}$} (D);
\draw[->] (C) to node[right,scale=0.8] {$H_I\otimes H_J$} (D);
\end{tz}
\end{rome}
Note that, since $P$ is a groupoid, the morphism $H_I$ is invertible in $P$, for all $I\subset \und{n}$, and hence $H$ is also invertible in $\KP\ord{n}$. This shows that $\KP\ord{n}$ is a groupoid. 

Given a morphism $s\colon \ord{n}\to \ord{m}$ in $\Gamma$, we define a functor 
\[ s^*\colon \KP\ord{n}\to \KP\ord{m}. \]
It sends an object $\{x_I,f_{I,J}\}$ in $\KP\ord{n}$ to the object $\{y_K,g_{K,L}\}$ of $\KP\ord{m}$ defined by
\begin{abc}
\item $y_K=x_{s^{-1}K}$, for all $K\subset \und{m}$,
\item $g_{K,L}=f_{s^{-1}K, s^{-1}L}$, for all disjoint $K,L\subset \und{m}$,
\end{abc} 
and a morphism in $\KP\ord{n}$ to the corresponding morphism in $\KP\ord{m}$. These data assemble into a functor $\KP\colon \Gamma\to \gpd$.

\subsection{Lack--Paoli's nerve}
We remind the reader of the definition of the Lack--Paoli $2$-nerve \cite{LackPaoli}. Let $\Nhom$ denote the $2$-category of bicategories, normal homomorphisms (i.e.,~pseudofunctors preserving the identities strictly) and icons (i.e.,~pseudonatural transformations between two functors that agree on objects preserving the identities strictly).

\begin{defn} 
The \emph{$2$-nerve} is defined to be the $2$-functor $N\colon \Nhom\to [\Dop, \cat]$ which sends a bicategory $\mathcal B$ to the simplicial object $\Nhom(-,\mathcal B)\colon \Dop\to \cat$ in categories. 
\end{defn}

For $n\geq 0$, we write $[n]$ for the ordered set $\{0,1,\ldots,n\}$, viewed as an object in $\Delta$. The category $\Nhom([n],\mathcal B)$ is the category of normal homomorphisms from $[n]$ to $\mathcal B$ and icons between them. The categories $[n]$ are here considered as bicategories with no non-identity $2$-morphisms, which gives a fully faithful inclusion of $\Delta$ into $\Nhom$. More explicitly, a normal homomorphism $(B,b,\beta)\colon [n]\to \mathcal B$ consists of 
\begin{abc}
\item for all $i\in [n]$, an object $B_i\in \mathcal{B}$,
\item for all $i\leq j$ in $[n]$, a morphism $b_{ij}\colon B_i\to B_j$ in $\mathcal B$,
\item for all $i\leq j\leq k$ in $[n]$, an invertible $2$-morphism $\beta_{ijk}\colon b_{ik} \to b_{jk}\circ b_{ij}$ in $\mathcal B$, 
\end{abc}
such that the following conditions hold: 
\begin{rome}
\item for all $i\in [n]$, the morphism $b_{ii}$ is the identity at $B_i$, 
\item for all $i\leq j$ in $[n]$, the isomorphisms $\beta_{iij}\colon  b_{ij}\cong b_{ij}\circ\id_{B_i}$ and $\beta_{ijj}\colon b_{ij}\cong \id_{B_j}\circ b_{ij}$ are the unit isomorphisms of $\mathcal B$, 
\item for all $i<j<k<l$ in $[n]$, the following diagram of $2$-morphisms commutes, 
\begin{tz}
\node (A) at (0,0) {$b_{il}$};
\node (B) at (0,-2.5) {$b_{kl}\circ b_{ik}$};
\node (C) at (3.5,0) {$b_{jl}\circ b_{ij}$};
\node (D) at (3.5,-2.5) {$b_{kl}\circ (b_{jk}\circ b_{ij})$};
\node (E) at (6,-1.25) {$(b_{kl}\circ b_{jk})\circ b_{ij}$};

\draw[->] (A) to node[left,scale=0.8] {$\beta_{ikl}$} (B);
\draw[->] (A) to node[above,scale=0.8] {$\beta_{ijl}$} (C);
\draw[->] (B) to node[below,scale=0.8] {$\id* \beta_{ijk}$} (D);
\draw[->] (D) to node[below,scale=0.8] {$\;\;\alpha$} (E);
\draw[->] (C) to node[above,scale=0.8] {$\;\;\;\;\;\;\;\;\;\beta_{jkl}* \id$} (E);
\end{tz}
where $\alpha$ denotes the associativity isomorphism of $\mathcal B$. Note that we omit the equality cases since these follow directly from the axioms of a bicategory and condition (ii). 
\end{rome}
Given another such normal homomorphism $(C,c,\gamma)\colon [n]\to \mathcal B$, an icon 
\[ \varphi\colon (B,b,\beta)\to (C,c,\gamma) \]
consists of 
\begin{abc}
\item for all $i\in [n]$, the satisfaction of the equation $B_i=C_i$, 
\item for all $i\leq j$ in $[n]$, a $2$-morphism $\varphi_{ij}\colon b_{ij}\to c_{ij}$, 
\end{abc}
such that the following conditions hold: 
\begin{rome}
\item for all $i\in [n]$, $\varphi_{ii}$ is the identity at $\id_{B_i}=\id_{C_i}$, 
\item for all $i<j<k$ in $[n]$, the diagram of $2$-morphisms below commutes.
\begin{tz}
\node (A) at (0,0) {$b_{ik}$};
\node (B) at (0,-2) {$c_{ik}$};
\node (C) at (2.5,0) {$b_{jk}\circ b_{ij}$};
\node (D) at (2.5,-2) {$c_{jk}\circ c_{ij}.$};

\draw[->] (A) to node[left,scale=0.8] {$\varphi_{ik}$} (B);
\draw[->] (A) to node[above,scale=0.8] {$\beta_{ijk}$} (C);
\draw[->] (B) to node[below,scale=0.8] {$\gamma_{ijk}$} (D);
\draw[->] (C) to node[right,scale=0.8] {$\varphi_{jk}*\varphi_{ij}$} (D);
\end{tz}
\end{rome}

Given a map $\alpha\colon [m]\to [n]$ in $\Delta$, the $2$-nerve construction induces a functor 
\[ \alpha^*\colon N\mathcal B[n]\to N\mathcal B[m]. \]
This functor sends an object $(B,b,\beta)$ in $N\mathcal B[n]$ to the object $\alpha^*(B,b,\beta)=(B',b',\beta')$ in $N\mathcal B[m]$ given by
\begin{abc}
\item $B'_r=B_{\alpha(r)}$, for all $r\in [m]$,
\item $b'_{rs}=b_{\alpha(r)\alpha(s)}\colon B_{\alpha(r)}\to B_{\alpha(s)}$, for all $r\leq s$ in $[m]$,
\item $\beta'_{rst}=\beta_{\alpha(r)\alpha(s)\alpha(t)}\colon b_{\alpha(r)\alpha(t)} \to b_{\alpha(s)\alpha(t)}\circ b_{\alpha(r)\alpha(s)}$, for all $r\leq s\leq t$ in $[m]$.
\end{abc}
This object $\alpha^*(B,b,\beta)$ satisfies conditions (i), (ii), and (iii) of being an object in $N\mathcal B[m]$, since the object $(B,b,\beta)$ in $N\mathcal B[n]$ does so. 

Moreover, the functor $\alpha^*$ sends a morphism $\varphi\colon (B,b,\beta)\to (C,c,\gamma)$ in $N\mathcal B[n]$ to the morphism $\alpha^*\varphi=\varphi'\colon \alpha^*(B,b,\beta)\to \alpha^*(C,c,\gamma)$ in $N\mathcal B[m]$ given by 
\begin{abc}
\item the fact that $B_{\alpha(r)}=C_{\alpha(r)}$, for all $r\in [m]$, since $B_i=C_i$ for all $i\in [n]$,
\item $\varphi'_{rs}=\varphi_{\alpha(r)\alpha(s)}\colon b_{\alpha(r)\alpha(s)}\to c_{\alpha(r)\alpha(s)}$, for all $r\leq s$ in $[m]$.
\end{abc}
The morphism $\alpha^*\varphi$ satisfies conditions (i) and (ii) of being a morphism in $N\mathcal B[m]$, since the morphism $\varphi$ in $N\mathcal B[n]$ does so.

\begin{ex}
The $2$-nerve of any bicategory finally provides us with examples of Tamsamani $2$-categories. Indeed, \cite{LackPaoli} show that all examples actually arise in this manner 
up to a $2$-equivalence. This indicates in particular that Tamsamani $n$-categories are a true generalization of known concepts of weak $n$-categories. 
\end{ex}

\subsection{Lack--Paoli's nerve for Picard categories} \label{NP}
A Picard category $(P,\otimes,\mathbf 1)$ can be seen as a bicategory with one object $*$ whose hom-category is $P$. 
Horizontal composition is given by the tensor product $\otimes$, and the unit and associativity isomorphisms are given by the ones of the 
monoidal structure on $P$. Note that every morphism and every $2$-morphism in this bicategory is invertible. 

Given a Picard category $(P,\otimes,\mathbf 1)$, we can apply Lack--Paoli's $2$-nerve to its associated bicategory $\mathcal P$. This gives a simplicial object in groupoids
\[ \NP =\Nhom(-,\mathcal P)\colon \Dop\to \gpd. \]
More explicitely, the category $\NP[0]=*$ is the terminal category and, for $n\geq 1$, an object in the category $\NP[n]$ is a collection $\{x_{ij}, f_{ijk}\}$ consisting of
\begin{abc}
\item for all $i\leq j$ in $[n]$, an object $x_{ij}\in P$, 
\item for all $i\leq j\leq k$ in $[n]$, an isomorphism $f_{ijk}\colon x_{ik}\to x_{jk}\otimes x_{ij}$ in $P$, 
\end{abc}
such that the following conditions hold: 
\begin{rome}
\item for all $i\in [n]$, $x_{ii}=\mathbf 1$, 
\item for all $i\leq j$ in $[n]$, the isomorphisms $f_{iij}\colon x_{ij}\cong x_{ij}\otimes \mathbf 1$ and $f_{ijj}\colon x_{ij}\cong \mathbf 1 \otimes x_{ij}$ are the unit isomorphisms of $P$, 
\item for all $i<j<k<l$ in $[n]$, the following diagram commutes, 
\begin{tz}
\node (A) at (0,0) {$x_{il}$};
\node (B) at (0,-2.5) {$x_{kl}\otimes x_{ik}$};
\node (C) at (3.5,0) {$x_{jl}\otimes x_{ij}$};
\node (D) at (3.5,-2.5) {$x_{kl}\otimes (x_{jk}\otimes x_{ij})$};
\node (E) at (6,-1.25) {$(x_{kl}\otimes x_{jk})\otimes x_{ij}$};

\draw[->] (A) to node[left,scale=0.8] {$f_{ikl}$} (B);
\draw[->] (A) to node[above,scale=0.8] {$f_{ijl}$} (C);
\draw[->] (B) to node[below,scale=0.8] {$\id\otimes f_{ijk}$} (D);
\draw[->] (D) to node[below,scale=0.8] {$\;\;\alpha$} (E);
\draw[->] (C) to node[above,scale=0.8] {$\;\;\;\;\;\;\;\;\;f_{jkl}\otimes \id$} (E);
\end{tz}
where $\alpha$ denotes the associativity isomorphism of $P$. 
\end{rome}
A morphism $H\colon \{x_{ij},f_{ijk}\}\to \{y_{ij},g_{ijk}\}$ is given by a collection of morphisms $\{H_{ij}\colon x_{ij}\to y_{ij}\}$ in $P$, indexed by $i\leq j$ in $[n]$, such that the following conditions hold: 
\begin{rome}
\item for all $i\in [n]$, $H_{ii}$ is the identity at $\mathbf 1$, 
\item for all $i<j<k$ in $[n]$, the following diagram commutes.
\begin{tz}
\node (A) at (0,0) {$x_{ik}$};
\node (B) at (0,-2) {$y_{ik}$};
\node (C) at (2.5,0) {$x_{jk}\otimes x_{ij}$};
\node (D) at (2.5,-2) {$y_{jk}\otimes y_{ij}$};

\draw[->] (A) to node[left,scale=0.8] {$H_{ik}$} (B);
\draw[->] (A) to node[above,scale=0.8] {$f_{ijk}$} (C);
\draw[->] (B) to node[below,scale=0.8] {$g_{ijk}$} (D);
\draw[->] (C) to node[right,scale=0.8] {$H_{jk}\otimes H_{ij}$} (D);
\end{tz}
\end{rome}
Note that, for all $i\leq j$ in $[n]$, the morphism $H_{ij}$ is invertible in $P$, since $P$ is a groupoid, and hence $H$ is also invertible in $\NP[n]$. This shows that $\NP[n]$ is a groupoid. 

Given a map $\alpha\colon [m]\to [n]$ in $\Delta$, the induced functor $\alpha^*\colon \NP[n]\to \NP[m]$ sends an object $\{x_{ij},f_{ijk}\}$ in $\NP[n]$, where $i\leq j\leq k$ range over $[n]$, to the object $\{x_{\alpha(r)\alpha(s)},f_{\alpha(r)\alpha(s)\alpha(t)}\}$ in $\NP[m]$, where $r\leq s\leq t$ range over $[m]$, and a morphism $H\colon \{x_{ij},f_{ijk}\}\to \{y_{ij},g_{ijk}\}$ in $\NP[n]$ to the morphism 
\[ \alpha^*H\colon \{x_{\alpha(r)\alpha(s)},f_{\alpha(r)\alpha(s)\alpha(t)}\}\to \{y_{\alpha(r)\alpha(s)},g_{\alpha(r)\alpha(s)\alpha(t)}\} \]
in $\NP[m]$ given by $(\alpha^*H)_{rs}=H_{\alpha(r)\alpha(s)}\colon x_{\alpha(r)\alpha(s)}\to y_{\alpha(r)\alpha(s)}$ for all $r\leq s$ in $[m]$.

\subsection{Comparison between $\KP$ and $\NP$} \label{CompareKN}
Let $(P,\otimes,\mathbf 1)$ be a Picard category. In \cref{KP}, we described the $K$-theory construction $\KP\colon \Gamma\to \gpd$ associated to $P$ and, in \cref{NP}, we described the $2$-nerve construction $\NP\colon \Dop\to \gpd$ applied to the bicategory $\mathcal P$ with one object $*$ associated to $P$. 

Recall that using the functor $\phi\colon \Dop\to \Gamma$, we can consider the underlying $K$-theory functor $U\KP\colon \Dop\to \gpd$ given by the composition of $\phi\colon \Dop\to \Gamma$ with $\KP\colon \Gamma \to \gpd$ as in \cref{DefUnderlyingDelta}. The aim of this subsection is to compare the functors $U\KP$ and $\NP$ in the following sense. 

\begin{thm} \label{KvsNerveThm}
There is a natural equivalence $U\KP\to \NP$ of functors $\Dsop{}\to \cat$. 
\end{thm}

\begin{proof}
The proof of the theorem occupies the remainder of this subsection. 
First recall that $\KP\ord{0}=*=\NP[0]$. Now let $n\geq 1$. Given an object $\{x_{ij},f_{ijk}\}$ in $\NP[n]$, we write it as
\[ \{ x_{[i+1,j]}, f_{[j+1,k]\sqcup [i+1,j]}\}, \]
where $[i+1,j]$ denotes the subset of $\und{n}$ containing all integers between $i+1$ and $j$. We can reformulate by saying that an object of $\NP[n]$ consists of a collection $\{x_I, f_{I,J}\}$, where $I\subset \und{n}$ runs over all convex subsets of $\und{n}$ and $I,J\subset \und{n}$ are disjoint convex subsets of $\und{n}$ such that $I\sqcup J$ is also convex. Similarly, a morphism in $\NP[n]$ consists of a collection $\{H_I\}$, where $I\subset \und{n}$ runs over all convex subsets of $\und{n}$. 

Given an object $\{ x_I, f_{I,J}\}$ in $\KP\ord{n}$, we can forget about all the data given by the non-convex subsets of $\und{n}$ and this gives rise to an object of $\NP[n]$. It is a routine exercise to check that the conditions are the same. Similarly, a morphism in $\KP\ord{n}$ gives rise to a morphism in $\NP[n]$ by forgetting the extra data. This gives a forgetful functor 
\[ {\mathcal U}_n\colon \KP\ord{n}\to \NP[n], \]
for every $n\geq 0$, where ${\mathcal U}_0=\id_*$. We show that these assemble into a natural transformation $\mathcal U\colon U\KP\Rightarrow \NP$ in $[\Dop,\gpd]$ such that each $\mathcal U_n$ is an equivalence of categories.

Let $\alpha\colon [m]\to [n]$ be a map in $\Delta$; we prove that the following square commutes.
\begin{tz}
\node (A) at (0,0) {$\KP\ord{n}$};
\node (B) at (0,-2) {$\KP\ord{m}$};
\node (C) at (2.5,0) {$\NP[n]$};
\node (D) at (2.5,-2) {$\NP[m]$};

\draw[->] (A) to node[left,scale=0.8] {$\phi(\alpha)^*$} (B);
\draw[->] (A) to node[above,scale=0.8] {$\mathcal U_n$} (C);
\draw[->] (B) to node[below,scale=0.8] {$\mathcal U_m$} (D);
\draw[->] (C) to node[right,scale=0.8] {$\alpha^*$} (D);
\end{tz}

Let $\{x_I, f_{I,J}\}$ be an object in $\KP\ord{n}$. By definition, 
\[ \mathcal U_m(\phi(\alpha)^*(\{x_I, f_{I,J}\}))=\{x_{\phi(\alpha)^{-1}([r+1,s])}, f_{\phi(\alpha)^{-1}([s+1,t]),\phi(\alpha)^{-1}([r+1,s])}\}, \]
where $I,J$ run over all disjoint subsets of $\und{n}$, and $r,s,t$ run over all $r\leq s\leq t$ in $[m]$. On the other hand, 
\begin{align*} 
\alpha^*(\mathcal U_n(\{x_I, f_{I,J}\}))&=\alpha^*(\{x_{[i+1,j]},f_{[j+1,k],[i+1,j]}\})\\ &=\{x_{[\alpha(r)+1,\alpha(s)]}, f_{[\alpha(s)+1,\alpha(t)],[\alpha(r)+1,\alpha(s)]}\} 
\end{align*}
where $I,J$ run over all disjoint subsets of $\und{n}$, while $i,j,k$ run over all $i\leq j\leq k$ in $[n]$, and $r,s,t$ run over all $r\leq s\leq t$ in $[m]$. But, for $r\leq s$ in $[m]$, we have that
\[ \phi(\alpha)^{-1}([r+1,s])=\bigsqcup_{l=r+1}^{s} [\alpha(l-1)+1,\alpha(l)]=[\alpha(r)+1, \alpha(s)]. \]
Hence $\mathcal U_m(\phi(\alpha)^*(\{x_I, f_{I,J}\}))=\alpha^*(\mathcal U_n(\{x_I, f_{I,J}\}))$ and we proceed similarly on morphisms. This shows the naturality of $\mathcal U$.

For each $n\geq 0$, let us now construct a functor $\mathcal F_n\colon \NP[n]\to \KP\ord{n}$ such that $\mathcal U_n$ and $\mathcal F_n$ form an equivalence of categories between $\KP\ord{n}$ and $\NP[n]$.

For $n=0$, we set $\mathcal F_0=\id_{*}$. For $n\geq 1$, we define $\mathcal F_n\colon \NP[n]\to \KP\ord{n}$ to be the functor which sends an object $\{x_{[i+1,j]}, f_{[j+1,k],[i+1,j]}\}$ in $\NP[n]$ to the object $\{\overline{x}_I, \overline{f}_{I,J}\}$ of $\KP\ord{n}$ given by
\begin{abc}
\item $\overline{x}_I=x_{[i_{p-1}+1,i_p]}\otimes (\cdots \otimes (x_{[i_3+1,i_4]}\otimes x_{[i_1+1,i_2]})\cdots)$, where $i_1<i_2<\ldots<i_{p-1}<i_p$ give the maximal decomposition of $I=[i_1+1,i_2]\sqcup [i_3+1,i_4]\sqcup \ldots \sqcup [i_{p-1}+1,i_p]$ into convex subsets, when $I\neq \emptyset$, and $\overline{x}_\emptyset=\mathbf{1}$, 

\item the morphisms $\overline{f}_{I,J}$ are induced by the morphisms $f_{[j+1,k],[i+1,j]}$ and the associativity and symmetry isomorphisms of $P$. 
\end{abc}

\begin{rmk}
We can also define $\overline{x}_I$ by induction on the number $p$ of convex subsets in the maximal decomposition of $I$. If $p=0$, then $I=\emptyset$ and set $x_\emptyset=\mathbf{1}$. If $p>0$ and $I=[i_1+1,i_2]\sqcup [i_3+1,i_4]\sqcup \ldots \sqcup [i_{p-1}+1,i_p]$ is the maximal decomposition of $I$ into convex subsets with $i_1<i_2<\ldots<i_{p-1}<i_p$, set $\overline{x}_I=x_{[i_{p-1}+1,i_p]}\otimes \overline{x}_{I\setminus [i_{p-1}+1,i_p]}$, where $I\setminus [i_{p-1}+1,i_p]$ has $p-1$ convex subsets in its maximal decomposition. 
\end{rmk}

A morphism $H$ in $\NP[n]$ is sent by $\mathcal F_n$ to the morphism $\overline{H}$ of $\KP\ord{n}$ given by
\[ \overline{H}_I=H_{[i_{p-1}+1,i_p]}\otimes (\cdots \otimes (H_{[i_3+1,i_4]}\otimes H_{[i_1+1,i_2]})\cdots),\]
where $i_1<i_2<\ldots<i_{p-1}<i_p$ give the maximal decomposition of $I=[i_1+1,i_2]\sqcup [i_3+1,i_4]\sqcup \ldots \sqcup [i_{p-1}+1,i_p]$ into convex subsets, when $I\neq \emptyset$, and $H_\emptyset=\id_{\mathbf{1}}$. This defines a functor $\mathcal F_n$ since the tensor product is functorial. 

Finally, we show that $\mathcal U_n$ and $\mathcal F_n$ form an equivalence of categories between $\KP\ord{n}$ and $\NP[n]$. Clearly, the composite
\[ \NP[n]\xrightarrow{\mathcal F_n}\KP\ord{n}\xrightarrow{\mathcal U_n}\NP[n] \]
is the identity. It remains to show that the composite
\[ \KP\ord{n}\xrightarrow{\mathcal U_n}\NP[n]\xrightarrow{\mathcal F_n}\KP\ord{n} \]
is isomorphic to the identity. We give a natural isomorphism
\[ \eta\colon \id_{\KP\ord{n}}\Longrightarrow \mathcal F_n\mathcal U_n. \]
At $\{x_I,f_{I,J}\}\in \KP\ord{n}$, the component of $\eta$ is given by the collection of morphisms
\[ f_{[i_{p-1}+1,i_p],\ldots,[i_1+1,i_2]}\colon x_I\to x_{[i_{p-1}+1,i_p]}\otimes (\cdots (\cdots \otimes x_{[i_1+1,i_2]})\cdots) \]
where $I=[i_1+1,i_2]\sqcup [i_3+1,i_4]\sqcup \ldots \sqcup [i_{p-1}+1,i_p]$ is the maximal decomposition of $I$ into convex subsets with $i_1<i_2<\ldots<i_{p-1}<i_p$. It is well-defined and natural, since it depends only on the morphisms $f_{I,J}$ and the associativity isomorphism. Moreover, it is a natural isomorphism since each component is an isomorphism. 

We have proven that the functors $\NP$ and $U\KP$ are naturally equivalent in $[\Dop, \cat]$.
\end{proof}

\section{Stable homotopy hypothesis for Picard categories}
\label{SHH_Picard}

The aim of this section is to recover a 
version of the stable homotopy hypothesis for Picard categories  \cite{Sinh}, \cite{DrinfeldInfDim}, \cite{BoyarchenkoNotes}, \cite{Patel}, \cite{JohnsonOsorno} from our more general result. 

We consider the category $\Pic$  
of Picard categories and strong symmetric monoidal functors between them which preserve the unit strictly, and the category $\PicGTam$ of Picard--Tamsamani ($1$-)categories and morphisms of $\Gamma$-groupoids between them. We show that the $K$-theory construction
\[ \mathcal{K}\colon \Pic\to \PicGTam, \]
introduced in \cref{KP}, induces an equivalence between the homotopy categories 
\[ \Ho(\Pic)\to \Ho(\PicGTam), \]
where $\Ho(\Pic)$ is the localization of $\Pic$ with respect to those morphisms in $\Pic$ whose underlying functor is an equivalence of groupoids, and $\Ho(\PicGTam)$ is the localization of $\PicGTam$ with respect to the levelwise equivalences in $[\Gamma,\gpd]$. Recall that, by \cref{we_PicTam_vsGammaGTam} in the case $n=1$, the levelwise equivalences of groupoids in $\PicGTam$ are exactly the weak equivalences in $\PicGTam$ as defined in \cref{equivPic}. 

Combining this equivalence of homotopy categories with the stable homotopy hypothesis for Picard--Tamsamani categories (this is \cref{StableHH} in the case $n=1$), we recover the stable homotopy hypothesis for Picard categories.

\begin{cor}\label{CorollaryPicard} 
The homotopy category $\Ho(\Pic)$ of Picard categories is equivalent to the stable homotopy category $\Ho(\spectra_{[0,1]})$ of stable $1$-types.
\end{cor}

\subsection{The homotopy inverse to $K$-theory}
\label{Construction_M}
Before starting the construction of the homotopy inverse which will occupy most of this subsection, we show that $\mathcal{K}$ actually takes values in $\PicGTam$. 

\begin{lem} \label{KtheoryPic}
The restriction of the $K$-theory functor to Picard categories takes values in Picard--Tamsamani $1$-categories. 
\end{lem}

\begin{proof}
From \cite{MandellKtheory}, we know that the $K$-theory construction takes values in $\Gamma$-categories. 
As we discussed already in \cref{KP}, the $K$-theory functor actually takes values in $\Gamma$-objects in groupoids if we start with a Picard category. 

Given a Picard category $(P,\otimes,\mathbf{1})$, it remains to show that $\KP$ is a very special $\Gamma$-object. Since being special only depends on the underlying simplicial object, as explained in \cref{PicTam_are_vsGammaGTam}, and is invariant under levelwise equivalences of simplicial objects, we conclude from \cref{KvsNerveThm} that $\KP$ is special, since $\NP$ is a Tamsamani $2$-category \cite{LackPaoli}. 

Finally, we need to check that $\pi_0\KP\ord{1}$ with the monoid structure from the Segal condition is a group. 
We observe that $\pi_0\KP\ord{1}\cong \pi_0 P$ and that the monoid structure on $\pi_0\KP\ord{1}$ is given by 
\[ \pi_0 \KP\ord{1}\times \pi_0 \KP\ord{1} \cong \pi_0 \KP\ord{2}\xrightarrow{\pi_0(m^*)} \pi_0 \KP\ord{1},\; (x_1,x_2)\mapsto x_1\otimes x_2. \]
Hence, $\pi_0 \KP\ord{1}$ is isomorphic as a monoid to $\pi_0 P$ with the monoid structure coming from $\otimes$. By definition of the Picard category $P$, the monoid $\pi_0 P$ is an abelian group, and thus so is $\pi_0\KP\ord{1}$. We conclude that $\KP$ is a Picard--Tamsamani category. 
\end{proof}

We now construct the homotopy inverse 
\[ \mathcal M\colon \PicGTam\to \Pic. \]
Our inspiration for this construction comes from the fact that the underlying simplicial objects of the K-theory and of Lack--Paoli’s $2$-nerve of a Picard 
category are levelwise equivalent, as shown in \cref{CompareKN}. With this in mind, we are able to upgrade the functor $G\colon \Tam^2\to\Nhom$, serving as an adjoint to the $2$-nerve in \cite[\textsection 7]{LackPaoli}, in order to incorporate the symmetry present in our context. Other variants of this construction are well-known in the literature. 

Let $\XX\colon \Gamma\to \gpd$ be a Picard--Tamsamani category. There is a natural way to construct a Picard category $\mathcal M \XX$ using the structure of the category $\Gamma$, where the underlying category is given by the groupoid $\XX\ord{1}$. More precisely, the symmetric monoidal structure on $\XX\ord{1}$ is defined as follows. 
\begin{itemize} 
\item Since the Segal map $S_2\colon \XX\ord{2}\to \XX\ord{1}\times \XX\ord{1}$ is an equivalence, we can choose an inverse $T\colon \XX\ord{1}\times \XX\ord{1}\to \XX\ord{2}$ for $S_2$ together with natural isomorphisms $\id_{\XX\ord{2}} \cong T\circ S_2$ and $\id_{\XX\ord{1}\times\XX\ord{1}}\cong S_2\circ T$. To be more concise in what follows, when we say that we choose an inverse $T$ for $S_2$, we admit that it also comes with the data of the two isomorphisms as above.
\item We define a functor $\otimes\colon \XX\ord{1}\times \XX\ord{1}\to \XX\ord{1}$ given by $\otimes=\XX(m)\circ T$, and an isomorphism $\sigma\colon \XX(m)\Rightarrow \otimes\circ S_2$ induced by $\id_{\XX\ord{2}} \cong T\circ S_2$, as in
\begin{tz}
\node (A) at (0,0) {$\XX\ord{2}$};
\node (B) at (2.5,0) {$\XX\ord{1}\times \XX\ord{1}$};
\node (C) at (2.5,-2) {$\XX\ord{1}$};
\draw[->] (A) to node[above,scale=0.8] {$S_2$} (B);
\draw[->] (A) to node[left,scale=0.8] {$\XX(m)\;\;$} (C);
\draw[->] (B) to node[right,scale=0.8] {$\otimes$} (C);

\node[scale=0.8] at (1.8,-.6) {$\sigma$};
\node at (1.8,-.9) {$\Rightarrow$};
\end{tz}
where $m\colon \ord{2}\to \ord{1}$ is such that $m(1)=m(2)=1$. This gives the monoidal product $\otimes\colon \XX\ord{1}\times \XX\ord{1}\to \XX\ord{1}$.
\item The monoidal unit
\[ u\colon *=\XX\ord{0} \xrightarrow{\XX(a)} \XX\ord{1} \]
is induced by the unique map $a\colon \ord{0}\to \ord{1}$.
\item The unitors $\lambda\colon \mathrm{id}\Rightarrow \otimes\circ (\mathrm{id},u)$ and $\rho\colon \mathrm{id}\Rightarrow \otimes\circ (u,\mathrm{id})$ are given by 
\begin{tz}
\node (Z) at (-2.5,0) {$\XX\ord{1}$};
\node (A) at (0,0) {$\XX\ord{2}$};
\node (B) at (2.5,0) {$\XX\ord{1}\times \XX\ord{1}$};
\node (C) at (2.5,-2) {$\XX\ord{1}$};
\draw[->] (A) to node[above,scale=0.8] {$S_2$} (B);
\draw[->] (A) to node[left,scale=0.8] {$\XX(m)\;\;$} (C);
\draw[->] (B) to node[right,scale=0.8] {$\otimes$} (C);
\draw[->] (Z) to [bend right=10] node[below,scale=0.8] {$\mathrm{id}$} (C);
\draw[->] (Z) to node[above,scale=0.8] {$\XX(\iota_1)$} (A);
\draw[->] (Z) to [bend left=30] node[above,scale=0.8] {$(\mathrm{id},u)$} (B);

\node[scale=0.8] at (1.8,-.6) {$\sigma$};
\node at (1.8,-.9) {$\Rightarrow$};

\node (Z) at (5,0) {$\XX\ord{1}$};
\node (A) at (7.5,0) {$\XX\ord{2}$};
\node (B) at (10,0) {$\XX\ord{1}\times \XX\ord{1}$};
\node (C) at (10,-2) {$\XX\ord{1}$};
\draw[->] (A) to node[above,scale=0.8] {$S_2$} (B);
\draw[->] (A) to node[left,scale=0.8] {$\XX(m)\;\;$} (C);
\draw[->] (B) to node[right,scale=0.8] {$\otimes$} (C);
\draw[->] (Z) to [bend right=10] node[below,scale=0.8] {$\mathrm{id}$} (C);
\draw[->] (Z) to node[above,scale=0.8] {$\XX(\iota_2)$} (A);
\draw[->] (Z) to [bend left=30] node[above,scale=0.8] {$(u,\mathrm{id})$} (B);

\node[scale=0.8] at (9.3,-.6) {$\sigma$};
\node at (9.3,-.9) {$\Rightarrow$};
\end{tz}
where $\iota_j\colon \ord{1}\to \ord{2}$ is such that $\iota_j(1)=j$ for $j=1,2$. 
\item For the symmetry isomorphism $\gamma\colon \otimes\circ\, \text{twist}\Rightarrow \otimes$, consider the following pasting diagrams 
\begin{tz}
\node (Y) at (0,2) {$\XX\ord{2}$};
\node (Z) at (2.5,2) {$\XX\ord{1}\times \XX\ord{1}$};
\node (A) at (0,0) {$\XX\ord{2}$};
\node (B) at (2.5,0) {$\XX\ord{1}\times \XX\ord{1}$};
\node (C) at (2.5,-2) {$\XX\ord{1}$};
\draw[->] (A) to node[above,scale=0.8] {$S_2$} (B);
\draw[->] (A) to node[left,scale=0.8] {$\XX(m)\;\;$} (C);
\draw[->] (B) to node[right,scale=0.8] {$\otimes$} (C);
\draw[->] (Y) to node[left,scale=0.8] {$\XX(\tau)$} (A);
\draw[->] (Y) to node[above,scale=0.8] {$S_2$} (Z);
\draw[->] (Z) to node[right,scale=0.8] {twist} (B);
\draw[->] (Z) to [bend left=60] node[right,scale=0.8] {$\otimes$} (C);

\node[scale=0.8] at (1.8,-.6) {$\sigma$};
\node at (1.8,-.9) {$\Rightarrow$};

\node[scale=0.8] at (3.2,-.4) {$\gamma$};
\node at (3.2,-.7) {$\Rightarrow$};

\node at (4.7,0) {$=$};

\node (A) at (5.5,1) {$\XX\ord{2}$};
\node (B) at (8,1) {$\XX\ord{1}\times \XX\ord{1}$};
\node (C) at (8,-1) {$\XX\ord{1}$};
\draw[->] (A) to node[above,scale=0.8] {$S_2$} (B);
\draw[->] (A) to node[left,scale=0.8] {$\XX(m)\;\;$} (C);
\draw[->] (B) to node[right,scale=0.8] {$\otimes$} (C);

\node[scale=0.8] at (7.3,.4) {$\sigma$};
\node at (7.3,.1) {$\Rightarrow$};
\end{tz}
where $\tau\colon \ord{2}\to \ord{2}$ is such that $\tau(1)=2$ and $\tau(2)=1$. Since we have chosen an inverse $T$ for $S_2$, we can define a unique invertible $\gamma\colon \otimes\circ \text{ twist}\Rightarrow \otimes$ using the data of $T$ such that the two pasting diagrams above are equal.
\item The associativity isomorphism $\alpha\colon \otimes \circ (\mathrm{id}\times \otimes) \Rightarrow \otimes \circ (\otimes \times \mathrm{id})$ is obtained in a similar way (see \cite{LackPaoli} for more details). 
\end{itemize}

One can then check that $\mathcal M \XX$ is a Picard category. The invertibility of the objects with respect to $\otimes$ comes from the fact that $\pi_0 \XX\ord{1}$ is an abelian group when endowed with the product
\[ \pi_0 \XX\ord{1}\times \pi_0 \XX\ord{1}\cong \pi_0 \XX\ord{2}\xrightarrow{\pi_0 \XX(m)} \pi_0 \XX\ord{1}, \]
where the first isomorphism is induced by the Segal map $S_2\colon \XX\ord{2}\to \XX\ord{1}\times \XX\ord{1}$.

Now let $F\colon \XX\to \YY$ be a morphism in $\PicGTam$. We define the strong symmetric monoidal functor $\mathcal M F\colon \mathcal M \XX\to \mathcal M \YY$ as follows. It is the functor $F_1\colon \XX\ord{1}\to \YY\ord{1}$ on the underlying categories. The compatibility isomorphism is given by the unique isomorphism $\psi\colon F_1\circ \otimes\Rightarrow \otimes'\circ (F_1\times F_1)$, induced by the data of the chosen inverses for $S_2$ and $S'_2$, 
such that
\begin{tz}
\node (A) at (-7.5,2) {$\XX\ord{2}$};
\node (B) at (-5,2) {$\XX\ord{1}\times \XX\ord{1}$};
\node (C) at (-5,0) {$\XX\ord{1}$};
\node (D) at (-5,-2) {$\YY\ord{1}$};
\draw[->] (A) to node[above,scale=0.8] {$S_2$} (B);
\draw[->] (A) to node[left,scale=0.8] {$\XX(m)\;\;$} (C);
\draw[->] (B) to node[right,scale=0.8] {$\otimes$} (C);
\draw[->] (C) to node[left,scale=0.8] {$F_1$} (D);
\draw[->] (B) to [bend left=60] node[right,scale=0.8] {$\otimes\circ (F_1\times F_1)$} (D);

\node[scale=0.8] at (-5.7,1.4) {$\sigma$};
\node at (-5.7,1.1) {$\Rightarrow$};

\node[scale=0.8] at (-4.3,-.4) {$\psi$};
\node at (-4.3,-.7) {$\Rightarrow$};

\node at (-1.25,0) {$=$};

\node (Y) at (0,2) {$\XX\ord{2}$};
\node (Z) at (2.5,2) {$\XX\ord{1}\times \XX\ord{1}$};
\node (A) at (0,0) {$\YY\ord{2}$};
\node (B) at (2.5,0) {$\YY\ord{1}\times \YY\ord{1}$};
\node (C) at (2.5,-2) {$\YY\ord{1}$};
\draw[->] (A) to node[above,scale=0.8] {$S'_2$} (B);
\draw[->] (A) to node[left,scale=0.8] {$\YY(m)\;\;$} (C);
\draw[->] (B) to node[right,scale=0.8] {$\otimes'$} (C);
\draw[->] (Y) to node[left,scale=0.8] {$F_2$} (A);
\draw[->] (Y) to node[above,scale=0.8] {$S_2$} (Z);
\draw[->] (Z) to node[right,scale=0.8] {$F_1\times F_1$} (B);

\node[scale=0.8] at (1.8,-.6) {$\sigma'$};
\node at (1.8,-.9) {$\Rightarrow$};
\end{tz}
This strong monoidal functor preserves the unit strictly since the diagram
\begin{tz}
\node (Y) at (0,2) {$*=\XX\ord{0}$};
\node (Z) at (3,2) {$\XX\ord{1}$};
\node (A) at (0,0) {$*=\YY\ord{0}$};
\node (B) at (3,0) {$\YY\ord{1}$};
\draw[->] (A) to node[below,scale=0.8] {$u'=\YY(a)$} (B);
\draw[->] (Y) to node[left,scale=0.8] {$\mathrm{id}$} (A);
\draw[->] (Y) to node[above,scale=0.8] {$u=\XX(a)$} (Z);
\draw[->] (Z) to node[right,scale=0.8] {$F_1$} (B);
\end{tz}
commutes by naturality of $F$. Moreover, it is compatible with the symmetry isomorphisms of $\mathcal M \XX$ and $\mathcal M \YY$. 

Since we have fixed a choice of inverse $T$ for the Segal map $S_2$ of each $\XX$, the compatibility isomorphism of a composite $G\circ F$ will be the composite of the compatibility isomorphisms of $F$ and $G$, respectively, for every two composable morphisms $F$, $G$ in $\PicGTam$. This comes from the fact that the compatibility isomorphisms are uniquely determined by the data of the chosen inverses. Hence, the construction $\mathcal M$ defines a functor $\mathcal M\colon \PicGTam\to \Pic$.  

Given a Picard category $P$, one can check that $\mathcal M \mathcal K P=\mathcal K P\ord{1} \cong P$, where the symmetric monoidal structures are isomorphic to each other. This gives a natural isomorphism $\mathcal M\mathcal K\cong \mathrm{id}_{\Pic}$ between endofunctors of $\Pic$. It remains to define a natural transformation 
\[ \zeta\colon \mathrm{id}_{\PicGTam}\Rightarrow \mathcal K\mathcal M \]
which is a levelwise equivalence in order to show that the functors $\mathcal K$ and $\mathcal M$ induce an equivalence between the homotopy categories of Picard categories and Picard--Tamsamani $1$-categories.

\subsection{The natural equivalence $\zeta\colon \XX\to \mathcal K\mathcal M \XX$}
We want to construct a natural transformation $\zeta\colon \XX\to \mathcal{K}\mathcal{M}\XX$, for every Picard--Tamsamani category $\XX$, which is levelwise an equivalence of groupoids. Recall that the levelwise equivalences are precisely the morphisms with respect to which we localize to obtain the homotopy category $\Ho(\PicGTam)$. We define the natural transformation $\zeta$ componentwise. For $\ord{n}\in \Gamma$, we need to define a functor
\[
\zeta_n\colon \XX\ord{n} \to (\mathcal{K}\mathcal{M} \XX)\ord{n}.
\]
This is almost a special case of \cite[\textsection 7]{LackPaoli}, but we need to take the symmetry into account.

\begin{notation}
For any $I \subset \und{n}$, denote by $\nu_I\colon \ord{n}\to \ord{1}$ the morphism in $\Gamma$ such that $\nu_I^{-1}(1)=I$. Note that $\nu_{\{i\}}$ is the map $\nu_i$ introduced in \cref{DefSpecial}, for each $i\in \und{n}$. For any disjoint $I,J\subset \und{n}$, denote by $\mu_{I,J} \colon \ord{n}\to \ord{2}$ the morphism in $\Gamma$ such that $\mu_{I,J}^{-1}(1)=I$ and $\mu_{I,J}^{-1}(2)=J$. Note that we have the relations $\nu_1\circ \mu_{I,J}=\nu_I$, $\nu_2\circ \mu_{I,J}=\nu_J$, and $m\circ \mu_{I,J}=\nu_{I\sqcup J}$.
\end{notation}

Assume now that we are given an object $x\in \XX\ord{n}$. We describe the value of 
\[ \zeta_n(x)=\{y_I, f_{I,J}\}\in \mathcal (K\mathcal M \XX)\ord{n}. \]
Set $y_I=\XX(\nu_I)(x) \in \XX\ord{1}$ for any $I\subset\und{n}$, and remember that the objects of $\XX\ord{1}$ are precisely the objects of $\mathcal{M}\XX$. Given two disjoint subsets $I,J\subset\und{n}$, we define the morphism
\[
f_{I,J}\colon y_{I\sqcup J} \to y_I \otimes y_J
\]
to be the morphism $f_{I,J}\colon \XX(\nu_{I\sqcup J})(x)\to \XX(\nu_I)(x)\otimes \XX(\nu_J)(x)$ given by the component at $x\in \XX\ord{n}$ of the $2$-morphism
\begin{tz}
\node(T) at (5.5,1) {$\XX\ord{n}$};
\node (A) at (8,1) {$\XX\ord{2}$};
\node (B) at (10.5,1) {$\XX\ord{1}\times \XX\ord{1}$};
\node (C) at (10.5,-1) {$\XX\ord{1}$.};
\draw[->] (A) to node[above,scale=0.8] {$S_2$} (B);
\draw[->] (A) to node[left,scale=0.8] {$\XX(m)\;\;$} (C);
\draw[->] (B) to node[right,scale=0.8] {$\otimes$} (C);
\draw[->] (T) to node[above,scale=0.8] {$\XX(\mu_{I,J})$} (A);
\draw[->] (T) to [bend right=10] node[below,scale=0.8] {$\XX(\nu_{I\sqcup J})\;\;\;\;\;\;\;\;$} (C);
\draw[->] (T) to [bend left=30] node[above,scale=0.8] {$(\XX(\nu_I),\XX(\nu_J))$} (B);

\node[scale=0.8] at (9.8,.4) {$\sigma$};
\node at (9.8,.1) {$\Rightarrow$};
\end{tz}
We need to show that this indeed defines an object in $(\mathcal{K}\mathcal{M} \XX)\ord{n}$. 

If we take $I=\emptyset$, then $\nu_\emptyset\colon \ord{n}\to \ord{1}$ factors through the unique map $a\colon \ord{0}\to \ord{1}$. Therefore 
\begin{tz}
\node (A) at (0,0) {$\XX\ord{n}$};
\node (B) at (3,0) {$\XX\ord{1}$};
\node (C) at (1.5,-1.5) {$*=\XX\ord{0}$};
\draw[->] (A) to node[above,scale=0.8] {$\XX(\nu_\emptyset)$} (B);
\draw[->] (A) to node[left,scale=0.8] {$!\;\;$} (C);
\draw[->] (C) to node[right,scale=0.8] {$\;u=\XX(a)$} (B);
\end{tz}
and $y_\emptyset=\XX(\nu_\emptyset)(x)=u$. Moreover, the map $\mu_{\emptyset,J}\colon \ord{n}\to \ord{2}$ factors as the composite 
\[ \ord{n}\xrightarrow{\nu_J}\ord{1}\xrightarrow{\iota_2}\ord{2}. \]
Therefore we have that the morphism $f_{\emptyset,J}\colon y_J\to u\otimes y_J$ is given by the component at $x\in \XX\ord{n}$ of the $2$-morphism 
\begin{tz}
\node(T) at (2.5,0) {$\XX\ord{n}$};
\node (Z) at (5,0) {$\XX\ord{1}$};
\node (A) at (7.5,0) {$\XX\ord{2}$};
\node (B) at (10,0) {$\XX\ord{1}\times \XX\ord{1}$};
\node (C) at (10,-2) {$\XX\ord{1}$};
\draw[->] (A) to node[above,scale=0.8] {$S_2$} (B);
\draw[->] (A) to node[left,scale=0.8] {$\XX(m)\;\;$} (C);
\draw[->] (B) to node[right,scale=0.8] {$\otimes$} (C);
\draw[->] (Z) to [bend right=10] node[below,scale=0.8] {$\mathrm{id}$} (C);
\draw[->] (Z) to node[above,scale=0.8] {$\XX(\iota_2)$} (A);
\draw[->] (Z) to [bend left=30] node[above,scale=0.8] {$(u,\mathrm{id})$} (B);
\draw[->] (T) to node[above,scale=0.8] {$\XX(\nu_J)$} (Z);

\node[scale=0.8] at (9.3,-.6) {$\sigma$};
\node at (9.3,-.9) {$\Rightarrow$};
\end{tz}
which is, by definition of $\mathcal M \XX$, the unit isomorphism $\rho_{y_J}\colon y_J\to u\otimes y_J$. Similarly we get $f_{I,\emptyset}=\lambda_{y_I}\colon y_I\to y_I\otimes u$.

It remains to show that the following triangle commutes
\begin{tz}
\node (A) at (0,0) {$y_{I\sqcup J}$};
\node (B) at (3,0) {$y_I\otimes y_J$};
\node (C) at (1.5,-1.5) {$y_J\otimes y_I$};
\draw[->] (A) to node[above,scale=0.8] {$f_{I,J}$} (B);
\draw[->] (A) to node[left,scale=0.8] {$f_{J,I}\;$} (C);
\draw[->] (C) to node[right,scale=0.8] {$\;\gamma$} (B);
\end{tz}
The map $\mu_{J,I}\colon \ord{n}\to \ord{2}$ factors as the composite
\[ \ord{n}\xrightarrow{\mu_{I,J}}\ord{2}\xrightarrow{\tau}\ord{2}. \]
Hence the compatibility with the symmetry follows from the equality of the two pasting diagrams. 
\begin{tz}
\node (T) at (-2.5,0) {$\XX\ord{n}$};
\node (Y) at (0,2) {$\XX\ord{2}$};
\node (Z) at (2.5,2) {$\XX\ord{1}\times \XX\ord{1}$};
\node (A) at (0,0) {$\XX\ord{2}$};
\node (B) at (2.5,0) {$\XX\ord{1}\times \XX\ord{1}$};
\node (C) at (2.5,-2) {$\XX\ord{1}$};
\draw[->] (A) to node[above,scale=0.8] {$S_2$} (B);
\draw[->] (A) to node[left,scale=0.8] {$\XX(m)\;\;$} (C);
\draw[->] (B) to node[right,scale=0.8] {$\otimes$} (C);
\draw[->] (Y) to node[left,scale=0.8] {$\XX(\tau)$} (A);
\draw[->] (Y) to node[above,scale=0.8] {$S_2$} (Z);
\draw[->] (Z) to node[right,scale=0.8] {twist} (B);
\draw[->] (T) to node[above,scale=0.8] {$\XX(\mu_{J,I})$} (A);
\draw[->] (T) to node[left,scale=0.8] {$\XX(\mu_{I,J})\;$} (Y);
\draw[->] (T) to [bend right=10] node[below,scale=0.8] {$\XX(\nu_{I\sqcup J})\;\;\;\;\;\;\;\;$} (C);
\draw[->] (Z) to [bend left=60] node[right,scale=0.8] {$\otimes$} (C);

\node[scale=0.8] at (1.8,-.6) {$\sigma$};
\node at (1.8,-.9) {$\Rightarrow$};

\node[scale=0.8] at (3.2,-.4) {$\gamma$};
\node at (3.2,-.7) {$\Rightarrow$};

\node at (4.7,0) {$=$};

\node(T) at (5.5,1) {$\XX\ord{n}$};
\node (A) at (8,1) {$\XX\ord{2}$};
\node (B) at (10.5,1) {$\XX\ord{1}\times \XX\ord{1}$};
\node (C) at (10.5,-1) {$\XX\ord{1}$};
\draw[->] (A) to node[above,scale=0.8] {$S_2$} (B);
\draw[->] (A) to node[left,scale=0.8] {$\XX(m)\;\;$} (C);
\draw[->] (B) to node[right,scale=0.8] {$\otimes$} (C);
\draw[->] (T) to node[above,scale=0.8] {$\XX(\mu_{I,J})$} (A);
\draw[->] (T) to [bend right=10] node[below,scale=0.8] {$\XX(\nu_{I\sqcup J})\;\;\;\;\;\;\;\;$} (C);

\node[scale=0.8] at (9.8,.4) {$\sigma$};
\node at (9.8,.1) {$\Rightarrow$};
\end{tz}

Next, we need to define the value of $\zeta_n$ on morphisms. Let $\varphi \colon x \to x'$ be a morphism in $\XX\ord{n}$. We define the map $\zeta_n(\varphi)$ by 
\[
\zeta_n(\varphi)_I=\XX(\nu_I)(\varphi) \colon \XX(\nu_I)(x) \to \XX(\nu_I)(x'), 
\]
for every $I\subset \und{n}$. First, note that $\XX(\nu_\emptyset)(\varphi)=\mathrm{id}_u$. Then, by naturality of $\sigma\colon \XX(m)\Rightarrow \otimes\circ S_2$, we have that the following diagram commutes.
\begin{tz}
\node (Y) at (0,2) {$\XX(\nu_{I\sqcup J})(x)=\XX(m)\XX(\mu_{I,J})(x)$};
\node (Z) at (7,2) {$\otimes\circ S_2(\XX(\mu_{I,J})(x))$};
\node (A) at (0,0) {$\XX(\nu_{I\sqcup J})(x')=\XX(m)\XX(\mu_{I,J})(x')$};
\node (B) at (7,0) {$\otimes\circ S_2(\XX(\mu_{I,J})(x'))$};
\draw[->] (A) to node[below,scale=0.8] {$\sigma_{\XX(\mu_{I,J})(x')}=f'_{I,J}$} (B);
\draw[->] (Y) to node[left,scale=0.8] {$\XX(\nu_{I\sqcup J})(\varphi)=\XX(m)\XX(\mu_{I,J})(\varphi)$} (A);
\draw[->] (Y) to node[above,scale=0.8] {$\sigma_{\XX(\mu_{I,J})(x)}=f_{I,J}$} (Z);
\draw[->] (Z) to node[right,scale=0.8] {$\otimes\circ S_2(\XX(\mu_{I,J})(\varphi))$} (B);
\end{tz}
where the morphism on the right corresponds to the morphism $\XX(\nu_I)(\varphi)\otimes \XX(\nu_J)(\varphi)$. This shows that $\zeta_n(\varphi)$ is indeed a morphism in $(\mathcal{K}\mathcal{M} \XX)\ord{n}$.

Moreover, this construction is 
compatible with composition and identity in $\XX\ord{n}$, since $\XX(\nu_I)$ is a functor from $\XX\ord{n}$ to $\XX\ord{1}$. Hence this defines a functor $\zeta_n\colon \XX\ord{n}\to (\mathcal K\mathcal M \XX)\ord{n}$ in $\gpd$.

We show that the functors $\zeta_n$, for $n\geq 0$, assemble into a natural transformation $\zeta\colon \XX \to \mathcal{K}\mathcal{M} \XX$ from $\Gamma$ to $\gpd$. Let $\rho\colon \ord{n} \to \ord{n'}$ be a morphism in $\Gamma$. Note that, for all disjoint $K,L\subset \und{n}'$, the following diagrams commute in $\Gamma$:
\begin{tz}
\node (A) at (0,0) {$\ord{n}$};
\node (B) at (3,0) {$\ord{n'}$};
\node (C) at (1.5,-1.5) {$\ord{1}$};
\draw[->] (A) to node[above,scale=0.8] {$\rho$} (B);
\draw[->] (A) to node[left,scale=0.8] {$\nu_{\rho^{-1}(K)}\;$} (C);
\draw[->] (B) to node[right,scale=0.8] {$\;\nu_K$} (C);

\node (A) at (6,0) {$\ord{n}$};
\node (B) at (9,0) {$\ord{n'}$};
\node (C) at (7.5,-1.5) {$\ord{2}$};
\draw[->] (A) to node[above,scale=0.8] {$\rho$} (B);
\draw[->] (A) to node[left,scale=0.8] {$\mu_{\rho^{-1}(K),\rho^{-1}(L)}\;$} (C);
\draw[->] (B) to node[right,scale=0.8] {$\;\mu_{K,L}$} (C);
\end{tz} 
Given an object $x \in \XX\ord{n}$, we have 
\begin{align*}
\zeta_{n'}(\XX(\rho)(x))&= (\XX(\nu_K)\XX(\rho)(x), \sigma_{\XX(\mu_{K,L})\XX(\rho)(x)})_{K\sqcup L\subset \underline{n}'} \\
&=(\XX(\nu_{\rho^{-1}(K)})(x),\sigma_{\XX(\mu_{\rho^{-1}(K),\rho^{-1}(L)})(x)})_{K\sqcup L\subset \underline{n}'} \\
&= (\mathcal{K}\mathcal{M} \XX)(\rho)(\zeta_n(x)),
\end{align*}
and similarly for morphisms in $\XX\ord{n}$. This shows that $\zeta\colon \XX\to \mathcal K\mathcal M \XX$ is a morphism in $\PicGTam$.

Moreover, each $\zeta_n\colon \XX\ord{n}\to (\mathcal K\mathcal M \XX)\ord{n}$ is an equivalence of groupoids. First note that $\zeta_1\colon \XX\ord{1}\to (\mathcal K\mathcal M \XX)\ord{1}$ is an isomorphism of categories. Since $\zeta$ is functorial in $\ord{n}$, the following diagram commutes
\begin{tz}
\node (Y) at (0,2) {$\XX\ord{n}$};
\node (Z) at (3.5,2) {$\XX\ord{1}^{\times n}$};
\node (A) at (0,0) {$\mathcal (K\mathcal M\XX)\ord{n}$};
\node (B) at (3.5,0) {$(K\mathcal M\XX)\ord{1}^{\times n}$};
\draw[->] (A) to node[below,scale=0.8] {$S_n$} node[above,scale=0.8] {$\simeq$} (B);
\draw[->] (Y) to node[left,scale=0.8] {$\zeta_n$} (A);
\draw[->] (Y) to node[above,scale=0.8] {$S_n$} node[below,scale=0.8] {$\simeq$} (Z);
\draw[->] (Z) to node[right,scale=0.8] {$\zeta_1^{\times n}$} node[left,scale=0.8] {$\cong$} (B);
\end{tz}
and hence $\zeta_n\colon \XX\ord{n}\to (\mathcal K\mathcal M \XX)\ord{n}$ is an equivalence of categories. 

Finally, we show that these morphisms $\zeta\colon \XX\to \mathcal K\mathcal M \XX$, for $\XX\in \PicGTam$, assemble into a natural transformation $\zeta\colon \mathrm{id}_{\PicGTam}\Rightarrow \mathcal K\mathcal M$. Let $F\colon \XX\to \YY$ be a morphism in $\PicGTam$. We want to show that, for each $n\geq 0$, the following diagram commutes.
\begin{tz}
\node (Y) at (0,2) {$\XX\ord{n}$};
\node (Z) at (3.5,2) {$\YY\ord{n}$};
\node (A) at (0,0) {$\mathcal (K\mathcal M\XX)\ord{n}$};
\node (B) at (3.5,0) {$(K\mathcal M\YY)\ord{n}$};
\draw[->] (A) to node[below,scale=0.8] {$(\mathcal K\mathcal M F)_n$} (B);
\draw[->] (Y) to node[left,scale=0.8] {$\zeta_{\XX,n}$} (A);
\draw[->] (Y) to node[above,scale=0.8] {$F_n$} (Z);
\draw[->] (Z) to node[right,scale=0.8] {$\zeta_{\YY,n}$} (B);
\end{tz}
First note that, by naturality of $F$, the following diagrams commute.
\begin{tz}
\node (A) at (0,0) {$\XX\ord{n}$};
\node (B) at (2.5,0) {$\YY\ord{n}$};
\node (C) at (0,-2) {$\XX\ord{1}$};
\node (D) at (2.5,-2) {$\YY\ord{1}$};
\draw[->] (A) to node[above,scale=0.8] {$F_n$} (B);
\draw[->] (A) to node[left,scale=0.8] {$\XX(\nu_I)$} (C);
\draw[->] (B) to node[right,scale=0.8] {$\YY(\nu_I)$} (D);
\draw[->] (C) to node[below,scale=0.8] {$F_1$} (D);

\node (A) at (5.5,0) {$\XX\ord{n}$};
\node (B) at (8,0) {$\YY\ord{n}$};
\node (C) at (5.5,-2) {$\XX\ord{2}$};
\node (D) at (8,-2) {$\YY\ord{2}$};
\draw[->] (A) to node[above,scale=0.8] {$F_n$} (B);
\draw[->] (A) to node[left,scale=0.8] {$\XX(\mu_{I,J})$} (C);
\draw[->] (B) to node[right,scale=0.8] {$\YY(\mu_{I,J})$} (D);
\draw[->] (C) to node[below,scale=0.8] {$F_2$} (D);
\end{tz} 
Let $x\in \XX\ord{n}$. Then 
\begin{align*}
\zeta_{\YY,n}F_n(x)&= (\YY(\nu_I)F_n(x),\sigma'_{\YY(\mu_{I,J})F_n(x)})_{I\sqcup J\subset \und{n}} \\
&= (F_1\XX(\nu_I)(x), \sigma'_{F_2\XX(\mu_{I,J})(x)})_{I\sqcup J\subset \und{n}} \\
&= (F_1\XX(\nu_I)(x), \psi_{(\XX(\nu_I)(x),\XX(\nu_J)(x))}\circ F_1(\sigma_{\XX(\mu_{I,J})(x)}))_{I\sqcup J\subset \und{n}} \\
&= (\mathcal K \mathcal M F)_n(\XX(\nu_I),\sigma_{\XX(\mu_{I,J})(x)})_{I\sqcup J\subset \und{n}} \\
&=(\mathcal K\mathcal M F)_n(\zeta_{\XX,n}(x))
\end{align*}
where $\sigma'_{F_2\XX(\mu_{I,J})(x)}=\psi_{(\XX(\nu_I)(x),\XX(\nu_J)(x))}\circ F_1(\sigma_{\XX(\mu_{I,J})(x)})$ by definition of $\psi$. A similar computation can be made for the morphisms in $\XX\ord{n}$.

This is the main part in the proof of the following result. 

\begin{thm} \label{PicPicTam}
The functors 
	\[\begin{tikzcd}
		\mathcal K\colon \Pic\ar[r, shift left=1.1ex] & \PicGTam\colon\mathcal M\ar[l, shift left=1.1ex] 
	\end{tikzcd}\]
induce an equivalence of homotopy categories 
	\[\begin{tikzcd}
		\Ho(\Pic)\ar[r, shift left=1.1ex] & \Ho(\PicGTam)\ar[l, shift left=1.1ex] 
	\end{tikzcd}\]
where $\Ho(\Pic)$ is the localization of $\Pic$ with respect to those morphisms in $\Pic$ whose underlying functor is an equivalence of groupoids, and $\Ho(\PicGTam)$ is the localization of $\PicGTam$ with respect to the levelwise equivalences in $[\Gamma,\gpd]$.
\end{thm}

\begin{proof}
We have recalled the constructions of the functors $\mathcal{K}$ in \cref{KP} and $\mathcal{M}$ in \cref{Construction_M}. We have checked that $\mathcal{K}$ takes values in $\PicGTam$ in \cref{KtheoryPic}. Next, we have to show that they induce functors on the homotopy categories above. Given a levelwise equivalence $F\colon \XX \to \YY$, the underlying functor of $\mathcal{M}F$ is precisely $F\ord{1}$ and thus an equivalence of groupoids. 

Conversely, given a morphism of Picard categories $P\to P'$ whose underlying functor is an equivalence of groupoids, we use the fact that both $\KP$, $\KP'$ are special $\Gamma$-objects in groupoids to conclude that the induced map $\KP\to \KP'$ is a levelwise equivalence in $[\Gamma,\gpd]$. 

As we have shown that there is a natural isomorphism $\mathcal{M}\mathcal{K}\cong\id_{\Pic}$ and that the natural transformation $\zeta\colon \id_{\PicGTam} \Rightarrow \mathcal{K}\mathcal{M}$ is a levelwise  equivalence, it completes the proof. 
\end{proof}

We are now ready to derive the promised corollary.

\begin{proof}[Proof of \cref{CorollaryPicard}]
This is now the straightforward combination of \cref{PicPicTam} and \cref{StableHH} in the case $n=1$. 
\end{proof}

\begin{rmk}
In addition to the equivalence of homotopy categories, \cite{LackPaoli} also show that their $2$-nerve induces a biequivalence of $2$-categories from $\Nhom$ to the $2$-category of Tamsamani $2$-categories if we allow pseudonatural transformations as the $1$-morphisms of the target. We expect an analogous statement to hold for Picard categories and Picard--Tamsamani $1$-categories.
\end{rmk}

\bibliographystyle{alpha}
\bibliography{StableTypes}
\end{document}